\documentclass[a4paper,11pt]{article}

\usepackage[T1]{fontenc}
\usepackage[english]{babel}
\usepackage{url}
\usepackage{amsmath,amssymb,amsthm}
\usepackage{dsfont,bbm}
\usepackage{pict2e}
\usepackage{multirow}
\usepackage{hyperref}
\usepackage{pgf,tikz}   
\usepackage{tikz-cd}   
\usepackage{float}
\usepackage[caption=false]{subfig} 
\usepackage{ifthen}\newboolean{color}
\usepackage{placeins}
\usetikzlibrary{positioning}
\usetikzlibrary{calc}
\usetikzlibrary{matrix}
\usepackage{siunitx}

\usepackage{graphicx}

\tikzset{%
	>=latex, 
	inner sep=0pt,%
	outer sep=2pt,%
	mark coordinate/.style={inner sep=0pt,outer sep=0pt,minimum size=3pt,
		fill=black,circle}%
}
\usepackage{listings}
\usepackage[shortlabels]{enumitem}
\theoremstyle{plain}

\newtheorem{theorem}{Theorem}[section]
\newtheorem{lemma}{Lemma}[section]
\newtheorem*{lem*}{Lemma}

\newtheorem*{cor*}{Corollary}
\newtheorem*{teo*}{Theorem}
\theoremstyle{definition}
\newtheorem{definition}{Definition}[section]
\newtheorem{ass}{Assumption}

\newtheorem{remark}{Remark}[section]

\newcommand{\numgru}{\num[group-minimum-digits=3]}
\newcommand{\STAB}[1]{\begin{tabular}{@{}c@{}}#1\end{tabular}}

 \usepackage[hmarginratio=1:1]{geometry}

\title{BDDC preconditioners for virtual element approximations of the three-dimensional Stokes equations}
\author{Tommaso Bevilacqua, Franco Dassi, Stefano Zampini and Simone Scacchi}

\begin{document}

\author{Tommaso Bevilacqua \footnote{University of Milan, Italy; e-mail: tommaso.bevilacqua@unimi.it} \and Franco Dassi \footnote{, University of Milan-Bicocca, Italy; e-mail: franco.dassi@unimib.it} \and Stefano Zampini \footnote{King Abdullah University of Science and Tecnology, Saudi Arabia; e-mail: stefano.zampini@kaust.edu.sa} \and Simone Scacchi \footnote{University of Milan, Italy; e-mail: simone.scacchi@unimi.it}}

\maketitle

\begin{abstract}
The Virtual Element Method (VEM) is a novel family of numerical methods for approximating partial differential equations on very general polygonal or polyhedral computational grids.
This work aims to propose a Balancing Domain Decomposition by Constraints (BDDC) preconditioner that allows using the conjugate gradient method to compute the solution of the saddle-point linear systems arising from the VEM discretization of the three-dimensional Stokes equations. We prove the scalability and quasi-optimality of the algorithm and confirm the theoretical findings with parallel computations. Numerical results with adaptively generated coarse spaces confirm the method's robustness in the presence of large jumps in the viscosity and with high-order VEM discretizations.\\

\end{abstract}

\maketitle

\section{Introduction}

The Virtual Element Method (VEM, \cite{beirao2013basic}) is a recent technology for the numerical approximation of Partial Differential Equations (PDEs) which can deal with computational grids of very general polygonal/polyhedral shape. Effective VEM discretizations have been developed for several PDEs; the interested reader should consult the recent special issue \cite{specialVEM.2021} and the book \cite{bookVEM.2022} for further details. Regarding computational fluid dynamics, divergence-free VEM discretizations of the Stokes and Navier-Stokes equations have been proposed in \cite{da2017divergence,da2018virtual,beirao2020stokes}.

Due to the arbitrary shape of polytopal elements, the linear systems arising from VEM discretizations of PDEs are generally severely ill-conditioned. Some recent studies have proposed multigrid and domain decomposition preconditioners for scalar elliptic equations in primal form: see \cite{antonietti2023agglomeration,antoniettiMasV.2018} for a multigrid preconditioner, \cite{bertoluzza2017bddc,bertoluzza2020,klawonn2022adaptive,klawonn2022three} for Balancing Domain Decomposition by Constraints (BDDC) and Dual-Primal Finite Element Tearing and Interconnecting (FETI-DP) preconditioners, and \cite{calvo.2018,calvo.2019} for Overlapping Additive Schwarz preconditioners.

A few works have investigated the efficient solution of VEM approximations for saddle point problems. In \cite{dassi2020parallel,dassiS.2020b} parallel block algebraic multigrid preconditioners have been proposed for three-dimensional VEM approximations of elliptic, Stokes, and Maxwell equations in mixed form. BDDC preconditioners for three-dimensional scalar elliptic equations in mixed form have been constructed and analyzed in \cite{dassiZS2022}. BDDC methods for two-dimensional VEM discretizations of the Stokes equations have been studied in \cite{bevilacqua2022bddc}.

The present study aims to construct, analyze and numerically validate a BDDC preconditioner for three-dimensional divergence-free VEM discretizations of the Stokes equations. The novelty with respect to our previous work \cite{bevilacqua2022bddc} is the extension of both analysis and implementation to the three-dimensional case, the inclusion of deluxe scaling functions to average the dual unknowns after local solves, the development of an adaptive technique to enrich the coarse space, and the parallel implementation of the algorithm. 
From the theoretical point of view, we provide a convergence estimate that yields the scalability and quasi-optimality of the method. We validate the theoretical estimates with several parallel numerical tests, and we show the robustness of the preconditioner with respect to the degree of approximation and the shape of the polyhedral elements. Finally, we compare the proposed algorithm with other parallel solvers in terms of computational efficiency, and we confirm the robustness of our approach on a challenging multi-sinker test case \cite{rudi2017weighted}.
\section{Continuous problem: the Stokes equations}
\label{sec:1}
We follow the standard notation for the Sobolev spaces as in \cite{notationBook}. Moreover we recall the differential operators: the vector Laplacian $\mathbf{\Delta}$, the divergence div, the gradient $\mathbf{\nabla}$ and the strain tensor $\varepsilon_{ij} (\mathbf{u}):= \frac{1}{2}\big( \frac{\partial u_i}{\partial x_j} + \frac{\partial u_j}{\partial x_i}\big)$.\\
We denote with $\mathcal{O}$ a generic geometrical entity (element, face, edge) having diameter $h_\mathcal{O}$ and we introduce for any $\mathcal{O}$ and $n \in \mathbb{N}$ the spaces: 
\begin{itemize}
	\item $\mathbb{P}_n(\mathcal{O})$ the set of polynomials on $\mathcal{O}$ of degree $\leq n$ (with $\mathbb{P}_{-1}(\mathcal{O}) := \{0\}$),
	\item $\widehat{\mathbb{P}}_{n\setminus m}(\mathcal{O}) := \mathbb{P}_n(\mathcal{O})\setminus \mathbb{P}_m(\mathcal{O})$ for $n>m$, denotes the space of polynomials in $\mathbb{P}_n(\mathcal{O})$ with monomials of degree strictly greater than $m$.
\end{itemize}  
Let $\Omega \subseteq \mathbb{R}^3$ be a bounded Lipschitz domain, with $\Gamma = \partial \Omega$, and consider the stationary Stokes problem on $\Omega$ with homogeneous Dirichlet boundary conditions: 
\begin{equation}\label{ContinuousProblem}
	\begin{cases}
		\text{Find }(\mathbf{u},p)\text{ such that} \\
		-\nu\mathbf\Delta\mathbf{u} - \nabla p = \mathbf{f} \qquad &\text{ in } \mathit{\Omega} \\
		\text{div}\,\mathbf{u} = 0  \qquad &\text{ in } \mathit{\Omega} \\
		\mathbf{u}=0 \qquad  &\text{ on } \mathit{\Gamma},
	\end{cases}
\end{equation}
where $\mathbf{u}$ and $p$ are respectively the velocity and the pressure fields, $\mathbf{f}\in [H^{-1}(\mathit{\Omega})]^3$ represents the external force and $\nu \in \mathbb{R}, \nu >0$ is the viscosity, assumed constant for the purposes of our analysis. Numerical evidence for the robustness of the algorithm in the presence of highly heterogeneous viscosity is provided in Section \ref{sec:numExe}.

Let us consider the spaces:
\begin{align}
	\mathbf{V}:=[H^{1}_{0}(\mathit{\Omega})]^3,\qquad Q:=L^{2}_{0}(\mathit{\Omega})=\bigg \{ q\in L^{2}(\mathit{\Omega})\quad s.t. \quad \int_{\mathit{\Omega}} q \text{ d}\Omega=0 \bigg \}.
\end{align}

Let the bilinear forms $a(\cdot,\cdot): \mathbf{V} \times \mathbf{V} \rightarrow \mathbb{R}$ and $b: \mathbf{V} \times Q \rightarrow \mathbb{R}$ be defined as:
\begin{equation}\label{aFormCont}
	a(\mathbf{u},\mathbf{v}) := \int_{\Omega}\nu \mathbf{\varepsilon}(\mathbf{u}) : \mathbf{\varepsilon}(\mathbf{v})\text{ d}\Omega \qquad \text{for all } \mathbf{u},\mathbf{v} \in \mathbf{V}
\end{equation}
\begin{equation}\label{bFormCont}
	b(\mathbf{v},q) := \int_{\Omega} \text{div} \, \mathbf{v}\, q \text{ d}\Omega \qquad \text{for all } \mathbf{u} \in \mathbf{V}, q \in Q.
\end{equation}
We recall here two lemmas (see \cite{klawonn2006} Section 2) that we will need in Section \ref{sec:convrate}, the first one guarantees the equivalence between the Stokes and elasticity bilinear forms and their $H^1$ seminorms, while the second one is a Korn-type inequality.
\begin{lemma}\label{kornineq}
    There exists a constant $c>0$ such that:
    \begin{align*}
        c\Vert \nabla\mathbf{u}\Vert_{[L^2(\Omega)]} \leq \Vert \varepsilon(\mathbf{u})\Vert_{[L^2(\Omega)]} \leq \Vert \nabla\mathbf{u}\Vert_{[L^2(\Omega)]}
        \qquad \forall \mathbf{u} \in [H^1(\Omega)]^3, \mathbf{u} \perp ker(\varepsilon),
    \end{align*}
    where $ker(\varepsilon)$ is the space of the rigid body modes of the elasticity problem.
\end{lemma}
\begin{lemma}\label{ineqminimizing}
    Let $\Omega \subset \mathbb{R}^3$ be a Lipschitz domain of diameter H and $\Sigma \subset \partial \Omega$ be an open subset with positive surface measure. There exists a positive constant $C$, independent of H such that:
    \begin{align*}
        \inf_{\mathbf{r} \in ker(\varepsilon)} \Vert \mathbf{u} - \mathbf{r} \Vert^2_{L^2(\Omega)} \leq C H |\mathbf{u}|_{E(\Sigma)} \qquad \forall \mathbf{u} \in [H^{1/2}(\Sigma)]^3,
    \end{align*}
    where $|\mathbf{u}|_{E(\Sigma)}:= \inf_{\mathbf{u}\in [H^1(\Omega)]^3,\mathbf{v}_{|\Sigma}= \mathbf{u}} \Vert \varepsilon (\mathbf{u}) \Vert_{L^2(\Omega)}$.
    
\end{lemma}
Then a standard variational formulation of problem \eqref{ContinuousProblem} reads:
\begin{equation}\label{VarForm}
	\begin{cases}
		\text{find } (\mathbf{u},p) \in \mathbf{V} \times Q \text{ such that} \\
		a(\mathbf{u},\mathbf{v})+b(\mathbf{v},p)=(\mathbf{f},\mathbf{v}) & \text{for all } \mathbf{v} \in \mathbf{V}, \\
		b(\mathbf{u},q)=0 & \text{for all } q \in Q,
	\end{cases}
\end{equation}
where
\begin{align*}
	(\mathbf{f},\mathbf{v}):=\int_{\Omega} \mathbf{f} \cdot \mathbf{v}\text{ d}\Omega.
\end{align*}
We refer to \cite{boffi2013mixed} for the mathematical analysis of this problem.

\section{Virtual element discretization}
\label{sec:2}
We present here the discretization of problem \eqref{ContinuousProblem}, based on the virtual element space introduced in Section 3 of \cite{beirao2020stokes}, that is designed to solve a Stokes-like problem element-wise.

\textbf{Mesh construction.} Let $\{\mathcal{T}_h\}_h$ be a sequence of decompositions of $\Omega$ into general shape-regular polyhedral elements $K$ (in sense of the Definition 2.1 of \cite{bertoluzza2020}) with a mesh size:
\begin{equation*}
	h:=\sup_{K\in \mathcal{T}_h}h_K,
\end{equation*}
where $h_K$ is the diameter of $K$.

We suppose that, for all $h$, each element $K\in \mathcal{T}_h$ satisfies the following assumptions:
\begin{itemize}
	\item$(\mathbf{A1})$ $K$ is star-shaped with respect to a ball $B_K$ of radius $\geq \gamma h_K$,
	\item $(\mathbf{A2})$ every face $f$ of  $K$ is star-shaped with respect to a disk $B_f$ of radius~$\geq~\gamma~h_K$,
	\item $(\mathbf{A3})$ every edge $e$ of $K$ satisfies $h_e \geq \gamma h_K$,
\end{itemize}
where $\gamma$ is a uniform positive constant. These hypotheses could be weakened as in~\cite{beirao2013basic}. For example, assuming that every $K$ is a union of a finite (and uniformly bounded) number of star-shaped domains, each satisfying ($\mathbf{A1}$).\\
We also assume that the scalar viscosity field $\nu$ is piecewise constant with respect to the decomposition $\mathcal{T}_h$, i.e., $\nu$ is constant on each polyhedron $K \in \mathcal{T}_h$.
We denote with $N_K, N_V, N_f$ and $N_e$ the total number of polyhedra, vertices, faces, and edges of the decomposition $\mathcal{T}_h$, respectively.\\
Since we are interested in estimates with a dependence on the number and size of the subdomains and polyhedral elements, we will denote with the symbol $\lesssim$ a bound up to a generic positive constant that is independent of the previous quantities, but which may depend on $\Omega$, on the polynomial order $k$ and the constant $\gamma$ of Assumptions $(\mathbf{A1})-(\mathbf{A3})$.\\
Before introducing the discrete velocity and pressure spaces, we define some suitable projection operators that will be directly computable given an accurate choice of the degrees of freedom (dofs). For any $n \in \mathbb{N}$ and each geometric entity $\mathcal{O}$ (element or face), we introduce the following polynomial projections:
\begin{itemize}
\item the $L^2$-$projection$ $\Pi^{0,\mathcal{O}}_n: L^2(\mathcal{O}) \rightarrow \mathbb{P}_n(\mathcal{O})$, defined for any $v \in L^2(\mathcal{O})$ by:
\begin{equation}
\int_{\mathcal{O}}q_n (v-\Pi^{0,\mathcal{O}}_n v) \text{ d}\mathcal{O} = 0	\qquad \text{for all } q_n \in \mathbb{P}_{n}(\mathcal{O}),
\end{equation}
with obvious extension for vector functions $\Pi^{0,\mathcal{O}}_n: [L^2(\mathcal{O})]^3 \rightarrow [\mathbb{P}_n(\mathcal{O})]^3$ and tensor functions $\Pi^{0,\mathcal{O}}_n: [L^2(\mathcal{O})]^{3 \times 3} \rightarrow [\mathbb{P}_n(\mathcal{O})]^{3 \times 3}$,
\item the $H^1$-$seminorm\ projection$ $\Pi^{\nabla,\mathcal{O}}_n: H^1(\mathcal{O}) \rightarrow \mathbb{P}_n(\mathcal{O})$, defined for any $v \in H^1(\mathcal{O})$ by:
\begin{equation}
	\begin{cases}
\displaystyle \int_{\mathcal{O}}\nabla q_n \cdot \nabla(v-\Pi^{\nabla,\mathcal{O}}_n v) \text{ d}\mathcal{O} = 0	\qquad \text{for all } q_n \in \mathbb{P}_{n}(\mathcal{O}),\vspace{0.2cm}\\
\displaystyle \int_{\partial \mathcal{O}}(v-\Pi^{\nabla,\mathcal{O}}_n v) \text{ d}\sigma = 0,
	\end{cases}	
\end{equation}
with obvious extension for vector functions $\Pi^{\nabla,\mathcal{O}}_n: [H^1(\mathcal{O})]^3 \rightarrow [\mathbb{P}_n(\mathcal{O})]^3$.
\end{itemize}

\textbf{Pressure space.} We start by constructing the discrete space $Q_h$. This is a natural extension of the two-dimensional space \cite{da2017divergence} and, following \cite{beirao2020stokes}, we define:
\begin{align}
	Q_h^K:=\mathbb{P}_{k-1}(K),
\end{align}
therefore the corresponding dofs are chosen defining for each $q \in Q_h(K)$ the following linear operator:
\begin{itemize}
	\item $\mathbf{D_Q}$: the moments up to order $k-1$ of $q$:
	\begin{align*}
		\int_{K}q p_{k-1} \text{ d}K	\qquad \text{for any } p_{k-1} \in \mathbb{P}_{k-1}(K).
	\end{align*}
\end{itemize}
The global space is given by:
\begin{align}\label{globdiscpre}
	Q_h:=\{q\in L^{2}_{0}(\mathit{\Omega}) \quad \text{s.t.}\quad q_{|K} \in Q^K_h\quad \text{for all } K\in \mathcal{T}_h\}.
\end{align}

\textbf{Velocity space.} The space $\mathbf{V}_h$, as defined in \cite{beirao2020stokes}, is the three-dimensional extension of the two-dimensional velocity space introduced in \cite{da2017divergence}, where the extensive use of the enhancement technique \cite{ahmad2013equivalent} is needed to achieve the computability of suitable polynomial projection operators.
We start by considering each face $f$ of a polyhedral element $K$, then we define:
\begin{equation}\label{faceSpace}
\begin{split}
\widehat{\mathbb{B}}_k(f) :=\{ v \in & H^1(f) \text{ s.t. (i) } v_{|e}\in C^0(\partial f), v_{|e}\in \mathbb{P}_k(e) \text{ for all } e \in \partial f,\\
&\text{(ii) } \Delta_f v \in \mathbb{P}_{k+1}(f),\\
&\text{(iii) } (v-\Pi^{\nabla,f}_k v,\widehat{p}_{k+1})_f = 0 \text{ for all } \widehat{p}_{k+1} \in \widehat{\mathbb{P}}_{k+1 \setminus k-2} (f) \}
\end{split}
\end{equation}
and the boundary space:
\begin{equation}\label{boundFace}
\widehat{\mathbb{B}}_k(\partial K) := \{v \in C^0(\partial K) \text{ s.t. } v_{|f}\in \widehat{\mathbb{B}}_k(f) \text{ for any } f \in \partial K \}.
\end{equation}
Then on the polyhedron $K$ we first define the virtual element space:
\begin{equation}\label{vemSpace}
\begin{split}
\widetilde{\mathbf{V}}_h^K:=\{\mathbf{v} \in & [H^1(K)]^3 \text{ s.t. (i) } \mathbf{v}_{|\partial K}\in [\widehat{\mathbb{B}}_k(\partial K)]^3, \\
& \text{(ii) }\Delta \mathbf{v} + \nabla s \in \mathbf{x} \land [\mathbb{P}_{k-1}(K)]^3 \text{ for some } s \in L^2_0(K), \\
& \text{(iii) div}\,  \mathbf{v} \in \mathbb{P}_{k-1}(K) \},
\end{split}
\end{equation} 
being $\mathbf{x} = (x_1,x_2,x_3)$ the independent variables. The velocity space is then defined as:
\begin{equation}\label{velSpace}
\begin{split}
\mathbf{V}_h^K:=\{ \mathbf{v} \in \widetilde{\mathbf{V}}_h^K & \text{ s.t. } (\mathbf{v}- \Pi^{\nabla,K}_k\mathbf{v}, \mathbf{x}\wedge \mathbf{\widehat{p}}_{k-1})_K=0 \\
& \forall \mathbf{\widehat{p}}_{k-1} \in [\widehat{\mathbb{P}}_{k-1\setminus k-3}(K)]^3 \}.
\end{split}
\end{equation}
\begin{remark}
The "super-enhanced" constraints (iii) in (\ref{faceSpace}) and in (\ref{velSpace}) are necessary to achieve the computability of the polynomial projection operators $\Pi^{0,f}_{k+1}$ and $\Pi^{0,K}_k$ (see Proposition 5.1 in \cite{beirao2020stokes}).
\end{remark}
\begin{remark}
Note that the approximation property is guaranteed by the fact that the spaces $\mathbf{V}_h^K$ and $Q_h^K$ contain $[\mathbb{P}_k(K)]^3$ and $\mathbb{P}_{k-1}(K)$, respectively.
\end{remark}
Given $\mathbf{v} \in \mathbf{V}_h^K$, the dofs of the local velocity space $\mathbf{V}_h^K$ are defined by means of the following set of linear operators: 
\begin{itemize}
\item $\mathbf{D^1_V}$: the values of $\mathbf{v}$ at the vertices of $K$;
\item $\mathbf{D^2_V}$: the values of $\mathbf{v}$ at $k-1$ distinct points of every edge $e$ of $K$;
\item $\mathbf{D^3_V}$: the face moments of $\mathbf{v}$ (split into normal and tangential components):
\begin{equation}
\int_f (\mathbf{v}\cdot \mathbf{n}_K^f)p_{k-2}\text{ d}f, \quad \int_f \mathbf{v}_\tau \cdot \mathbf{p}_{k-2}\text{ d}f,
\end{equation}
for all $p_{k-2}\in \mathbb{P}_{k-2}(f)$ and $\mathbf{p}_{k-2} \in [\mathbb{P}_{k-2}(f)]^2$, where $\mathbf{n}_K^f$ is the normal vector associated to the face $f$ and $\mathbf{v}_\tau$ is the 2D vector field defined on $\partial K$, s.t. on each face $f \in \partial K$:
\begin{align*}
    \mathbf{v}_{\tau}:=\mathbf{v}-(\mathbf{v} \cdot \mathbf{n}_K^f)\mathbf{n}_K^f;
\end{align*}
\item $\mathbf{D^4_V}$: the volume moments of $\mathbf{v}$:
\begin{equation}
\int_K \mathbf{v} \cdot (\mathbf{x} \wedge \mathbf{p}_{k-3})\text{ d}K \qquad \text{ for all } \mathbf{p}_{k-3} \in [\mathbb{P}_{k-3}(K)]^3;
\end{equation}
\item $\mathbf{D^5_V}$: the volume moments of $\text{div} \, \mathbf{v}$:
\begin{equation}
\int_K \text{div} \, \mathbf{v}\,\widehat{p}_{k-1}\text{ d}K \qquad \text{ for all } \widehat{p}_{k-1} \in \widehat{\mathbb{P}}_{k-1\setminus 0}(K).
\end{equation}
\end{itemize}
The global space $\mathbf{V}_h$ is obtained by gluing the local spaces:
\begin{equation}\label{globVel}
\mathbf{V}_h:=\{ \mathbf{v} \in [H^1(\Omega)]^3 \text{ s.t. } \mathbf{v}_{|K} \in  \mathbf{V}_h^K \text{ for all } K \in \mathcal{T}_h \}.
\end{equation}

\subsection{Discrete bilinear forms and load term approximation}
\label{subsec:3.1}
Here we can now discuss the discretization of the bilinear forms defined in (\ref{VarForm}). 
First, we decompose into local contribution the bilinear forms $a(\cdot,\cdot)$ and $b(\cdot,\cdot)$ and the load term~$\mathbf{f}$:
\begin{equation}\label{contForm}
a(\mathbf{u},\mathbf{v}):=\sum_{K \in \mathcal{T}_h} a^K(\mathbf{u},\mathbf{v}), \quad b(\mathbf{v},p):=\sum_{K \in \mathcal{T}_h} b^K(\mathbf{v},p), \quad (\mathbf{f},\mathbf{v}):=\sum_{K \in \mathcal{T}_h} (\mathbf{f},\mathbf{v})_K,
\end{equation}
for all $\mathbf{u}, \mathbf{v} \in [H^1(\Omega)]^3$. \\

We note that we do not need any approximation for the divergence bilinear form since we can compute exactly $b(\mathbf{v}_h,q_h)$ for all $\mathbf{v}_h \in \mathbf{V}_h$ and $q_h \in Q_h$ directly form the $\mathbf{D^1_V}, \mathbf{D^2_V}$ and $\mathbf{D^5_V}$.

Instead, the bilinear form $a(\cdot,\cdot)$ is not directly computable from the dofs when both entries are "virtual". Following \cite{beirao2020stokes}, we define the approximation:
\begin{equation}\label{discreteA}
a^K_h(\mathbf{u},\mathbf{v}):=\int_K(\Pi^{0,K}_{k-1}\varepsilon(\mathbf{u})):(\Pi^{0,K}_{k-1}\varepsilon(\mathbf{v})) \text{ d}K + S^K((I-\Pi^{\nabla,K}_k)\mathbf{u},(I-\Pi^{\nabla,K}_k)\mathbf{v}),
\end{equation}
for all $\mathbf{u}, \mathbf{v} \in \mathbf{V}_h^K$, where:
\begin{align*}
    \Pi^{0,K}_{k-1}\varepsilon(\mathbf{u})= \frac{\Pi^{0,K}_{k-1}\nabla \mathbf{u}+(\Pi^{0,K}_{k-1}\nabla \mathbf{u}))^T}{2}.
\end{align*}
The approximate bilinear form \eqref{discreteA} is obtained as the sum of two contributions, the first term known as the consistency part and the second term known as the stabilization part, where $S^P: \mathbf{V}_h^K \times \mathbf{V}_h^K \rightarrow \mathbb{R}$ is a suitable symmetric bilinear form that has to scale like the $H^1$-seminorm.
\begin{remark}
For the numerical experiments in Section~\ref{sec:numExe}, we use the $D$-recipe stabilization introduced in Section 6 of \cite{beirao2020stokes}.
\end{remark}
The load term is approximated by taking:
\begin{equation}\label{discreteF}
(\mathbf{f}_h,\mathbf{v})_K := \int_K \Pi^{0,K}_{k}f \cdot \mathbf{v} \text{ d}K.
\end{equation}
Finally, the global forms are obtained by simply gluing elements' contributions:
\begin{equation}\label{globDiscForm}
a_h(\mathbf{u},\mathbf{v}):=\sum_{K \in \mathcal{T}_h} a_h^K(\mathbf{u},\mathbf{v}), \quad  (\mathbf{f}_h,\mathbf{v}):=\sum_{K \in \mathcal{T}_h} (\mathbf{f}_h,\mathbf{v})_K,
\end{equation}
for all $\mathbf{u}, \mathbf{v} \in \mathbf{V}_h$.

\subsection{Discrete problem}
\label{subsec:3.2}
Using the discrete spaces (\ref{globVel}) and (\ref{globdiscpre}) and the discrete linear and bilinear forms previously introduced, the discrete Stokes problem reads as follows: 
\begin{equation}\label{DiscProblem}
	\begin{cases}
		\text{find } (\mathbf{u}_h,p_h) \in \mathbf{V}_{h,0} \times Q_{h,0} \text{ such that} \\
		a_h(\mathbf{u}_h,\mathbf{v}_h)+b(\mathbf{v}_h,p_h)=(\mathbf{f}_h,\mathbf{v}_h) & \text{for all } \mathbf{v}_h \in \mathbf{V}_{h,0}, \\
		b(\mathbf{u}_h,q_h)=0 & \text{for all } q_h \in Q_{h,0},
	\end{cases}
\end{equation}
where $\mathbf{V}_{h,0} := \mathbf{V}_h \cap [H^1_0(\Omega)]^3$ and $Q_{h,0} := Q_h \cap L^2_0(\Omega)$.\\
Combining the arguments in \cite{da2017divergence}, \cite{da2018virtual} and \cite{brenner2018virtual}, it is possible to show that the virtual space $\mathbf{V}_h$ has an optimal interpolation operator (see Lemma \ref{interpST}) and that the pair $(\mathbf{V}_h,Q_h)$ is inf-sup stable with $\beta_h >0$.\\
We have the following existence and convergence theorem that extends the analogous result for the two-dimensional case (\cite{da2017divergence}).

\begin{theorem}
Under the Assumptions $(\mathbf{A1})-(\mathbf{A3})$, let $(\mathbf{u},p)\in [H^1_0(\Omega)] \times L^2_0(\Omega)$ be the solution of the problem (\ref{ContinuousProblem}) and $(\mathbf{u}_h,p_h)\in \mathbf{V}_{h,0} \times Q_{h,0}$ be the unique solution of the problem (\ref{DiscProblem}). Assuming moreover $\mathbf{u}, \mathbf{f} \in [H^{s+1}(\Omega)]^3$ and $p \in H^s(\Omega)$, $0 < s \leq k$, then:
\begin{equation}\label{teoConv}
\begin{split}
&|\mathbf{u}-\mathbf{u}_h|_1 \lesssim h^s \mathcal{F}(\mathbf{u},\nu) + h^{s+2} \mathcal{H}(\mathbf{f},\nu),\\
&\Vert p-p_h \Vert_0 \lesssim h^s|p|_s + h^s \mathcal{K}(\mathbf{u},\nu) + h^{s+2} |\mathbf{f}|_{s+1},
\end{split}
\end{equation} 
for suitable functions $\mathcal{F}, \mathcal{H}, \mathcal{K}$ independent of $h$.
\end{theorem}

\begin{remark}
Since the error of the velocity in (\ref{teoConv}) does not depend on the pressure, one can design a reduced scheme with a smaller number of dofs as for the two-dimensional case (Section 5.3 in \cite{beirao2020stokes}).
\end{remark}

\section{BDDC preconditioner}\label{sec::4}
BDDC preconditioners \cite{dohrmann2003preconditioner} belong to a class of non-overlapping domain decomposition methods. They have been extensively applied to solve linear systems that arise from finite element discretizations of PDEs (see e.g. \cite{klawonn2006, klawonn2006dual,li2006bddc, zampini2017multilevel}), and recently they have also been extended to VEM discretizations \cite{bertoluzza2017bddc,bertoluzza2020,bevilacqua2022bddc,klawonn2022adaptive,klawonn2022three}. Here, we apply them to solve the saddle point linear system arising from the previous VEM discretization of three-dimensional Stokes equations. 

\subsection{Domain decomposition}
We decompose $\mathcal{T}_h$ into $N$ non-overlapping subdomains $\Omega_i$ with characteristic size $H_i$:
\begin{equation}
\bar{\mathcal{T}}_h = \bigcup_{i=1}^N \bar{\Omega}_i, \qquad \Gamma = \bigcup_{i\neq j} \partial \Omega_i \cap \partial \Omega_j,
\end{equation} 
where each $\Omega_i$ is union of different polyhedra of the tassellation $\mathcal{T}_h$ and $\Gamma$ is the interface (skeleton) among the subdomains. \\
We assume that the decomposition is shape-regular in the sense of \cite{bertoluzza2020}:\\
There exist a constant $\gamma^\star  > 0$ and $N^\star  > 0$ such that the subdomain decomposition satisfies the following properties:
\begin{itemize}
    \item it is geometrically conforming, that is, for all $i$, if a vertex, edge, or face of
$\Omega_i$  is contained in $\partial \Omega_i  \cap  \partial \Omega_j$, it is also, respectively, a vertex, edge, or face of $\Omega_j$ ;
\item the subdomains $\Omega_i$  are shape regular of diameter $H_i$  with constants $\gamma  \Omega_i> \gamma^\star$  and $N_{\Omega_i}  < N_\star$;
\item for all $i$  , there exists a scalar $\rho_i> 0$ such that $\rho |_{\Omega_i}  \simeq  \rho_i$;
\item the decomposition is quasi-uniform: there exists an H such that for all $i$  we
have $H_i  \simeq H$.
\end{itemize}
We will refer to the edges and faces of the subdomains $\Omega_i$ as macro edges and
macro faces. We let $\mathcal{E}_H$, and $\mathcal{F}_H$ denote, respectively, the set of macro edges E and of macro faces F of the subdomain decomposition interior to $\Omega$, and $\mathcal{F}^i_H$ and $\mathcal{E}^i_H$ denote the set of, respectively, macro faces and macro edges of the subdomain $\Omega^i$. 

\subsection{Decomposition of VEM spaces}
The discrete variational problem can be written as:
\begin{equation}\label{MatrixForm}
	\left[
	\begin{array}{cc}
		A & B^T\\
		B & 0\\
	\end{array}
	\right]
	\left[
	\begin{array}{c}
		\mathbf{u}\\
		p\\
	\end{array}
	\right]
	=
	\left[
	\begin{array}{c}
		\mathbf{f}\\
		0\\
	\end{array}
	\right],
\end{equation}
where the matrices A and B are associated with the discrete bilinear forms $a_h(\cdot,\cdot)$ and $b(\cdot,\cdot)$.
In the remainder of the paper, we omit the underscore $h$ since we will always refer to the finite-dimensional space. We write $\mathbf{\widehat{V}}\times Q$ instead of $\mathbf{\widehat{V}}_{h,0}\times Q_{h,0}$, only for the sake of simplifying the notation.
We split the velocity components' degrees of freedom (dofs) into boundary and interior dofs. In particular, all the dofs $\mathbf{D}^4_V$ and $\mathbf{D}^5_V$ are classified as interior dofs, while the $\mathbf{D}^1_V$, $\mathbf{D}^2_V$ and $\mathbf{D}^3_V$ are split into dofs that belongs to a single subdomain $\Omega_i$ (internal) or that belong to more than a single subdomain (boundary).
Following the notations introduced in \cite{li2006bddc} and \cite{bevilacqua2022bddc}, we decompose the discrete velocity and pressure space $\mathbf{\widehat{V}}$ and $Q$ into:
\begin{equation}\label{discVQ}
	\mathbf{\widehat{V}} = \mathbf{V}_I \bigoplus \mathbf{\widehat{V}}_\Gamma\text{,}
	\quad Q = Q_I\bigoplus Q_0\text{,}
\end{equation}
with $Q_0:=\prod_{i=1}^N \{q\in \Omega_i | q \textit{ is constant in } \Omega_i\}$.\\
$\mathbf{\widehat{V}}_\Gamma$ is the continuous space of the traces on $\Gamma$ of functions in $\mathbf{\widehat{V}}$, $\mathbf{V}_I$ and $Q_I$ are direct sums of subdomain interior velocity spaces $\mathbf{V}_I^{(i)}$, and subdomain interior pressure spaces $Q_I^{(i)}$, respectively, i.e.,
\begin{equation}\label{discViQi}
	\mathbf{V}_I = \bigoplus_{i=1}^{N} \mathbf{V}_I^{(i)}\text{,}\quad Q_I = \bigoplus_{i=1}^{N} Q_I^{(i)}\text{.}
\end{equation}
We also define the space of interface velocity variables of the subdomain $\Omega_i$ by $\mathbf{V}_\Gamma^{(i)}$, and the associated product space by $\mathbf{V}_\Gamma = \prod_{i=1}^{N}\mathbf{V}_\Gamma^{(i)}$; generally functions in $\mathbf{V}_\Gamma$ are discontinuous across the interface.
With the decomposition of the solution space given in \eqref{discVQ}, the global saddle-point problem \eqref{MatrixForm} can be written as: find $(\mathbf{u}_I,p_I,\mathbf{u}_\Gamma,p_0) \in (\mathbf{V}_I,Q_I,\mathbf{\widehat{V}}_\Gamma,Q_0)$, such that:
\begin{equation}\label{discMat}
	\left[
	\begin{array}{cccc}
		A_{II} & B_{II}^T & \widehat{A}_{\Gamma I}^T & 0\\
		B_{II} & 0 & \widehat{B}_{I\Gamma} & 0\\
		\widehat{A}_{\Gamma I} & \widehat{B}_{I\Gamma}^T & \widehat{A}_{\Gamma \Gamma} & \widehat{B}_{0\Gamma}^T\\
		0 & 0 & \widehat{B}_{0\Gamma}^T & 0\\
	\end{array}
	\right]
	\left[
	\begin{array}{c}
		\mathbf{u}_I\\
		p_I\\
		\mathbf{u}_\Gamma\\
		p_0\\
	\end{array}
	\right]
	=
	\left[
	\begin{array}{c}
		\mathbf{f}_I\\
		0\\
		\mathbf{f}_\Gamma\\
		0\\
	\end{array}
	\right].
\end{equation}
The blocks related to the continuous interface velocity are assembled from the corresponding subdomain submatrices, e.g., $\widehat{A}_{\Gamma \Gamma} = \sum_{i=1}^{N} {R_\Gamma^{(i)}}^T \widehat{A}_{\Gamma \Gamma}^{(i)} R_\Gamma^{(i)}$ and $\widehat{B}_{0 \Gamma} = \sum_{i=1}^{N} \widehat{B}_{0 \Gamma}^{(i)} R_\Gamma^{(i)}$. Correspondingly, the right-hand side vector $\mathbf{f}_I$ consists of subdomain vectors $\mathbf{f}_I^{(i)}$, and $\mathbf{f}_\Gamma$ is assembled from the subdomain components $\mathbf{f}_\Gamma^{(i)}$; we denote the spaces of the right-hand side vectors $\mathbf{f}_I$ and $\mathbf{f}_\Gamma$ by $\mathbf{F}_I$ and $\mathbf{F}_\Gamma$ respectively.

We proceed to eliminate, by static condensation, the independent subdomain variables $(\mathbf{u}_I,p_I)$ solving independent Dirichlet problems: 
\begin{equation}\label{uipi}
	\left[
	\begin{array}{cc}
		A_{II} & B_{II}^T\\
		B_{II} & 0
	\end{array}
	\right]
	\left[
	\begin{array}{c}
		\mathbf{u}_I\\
		p_I\\
	\end{array}
	\right] + 
	\left[
	\begin{array}{cc}
		\widehat{A}_{\Gamma I}^T & 0\\
		\widehat{B}_{I\Gamma} & 0
	\end{array}
	\right]
	\left[
	\begin{array}{c}
		\mathbf{u}_\Gamma\\
		p_0\\
	\end{array}
	\right] = 
	\left[
	\begin{array}{c}
		\mathbf{f}_I\\
		0\\
	\end{array}
	\right] \text{,}
\end{equation}
and obtain the global interface saddle point problem:
\begin{equation}\label{globInt}
	\widehat{S}\text{ }\widehat{u} = 
	\left[
	\begin{array}{cc}
		\widehat{S}_\Gamma & {\widehat{B}_{0\Gamma}}^T\\
		\widehat{B}_{0\Gamma} & 0\\
	\end{array}
	\right]
	\left[
	\begin{array}{c}
		\mathbf{u}_\Gamma\\
		p_0\\
	\end{array}
	\right]
	=
	\left[
	\begin{array}{c}
		\mathbf{g}_\Gamma\\
		0\\
	\end{array}
	\right]
	= \widehat{\mathbf{g}}\text{,}
\end{equation}
where the right-hand side $\widehat{\mathbf{g}}\in \mathbf{F}_\Gamma \times F_0$ is given by
\begin{align}
	\widehat{\mathbf{g}} = \sum_{i=1}^{N} {R_\Gamma^{(i)}}^T \bigg \{ 
	\left[
	\begin{array}{c}
		\mathbf{f}_\Gamma^{(i)}\\
		0\\
	\end{array}
	\right] - 
	\left[
	\begin{array}{cc}
		A_{\Gamma I}^{(i)} & {B_{I\Gamma}^{(i)}}^T\\
		0 & 0
	\end{array}
	\right]
	\left[
	\begin{array}{cc}
		A_{II}^{(i)} & {B_{II}^{(i)}}^T\\
		B_{II}^{(i)} & 0
	\end{array}
	\right]^{-1}
	\left[
	\begin{array}{c}
		\mathbf{f}_I^{(i)}\\
		0\\
	\end{array}
	\right]
	\bigg \}.
\end{align}
Here as in \cite{bevilacqua2022bddc}, $R_\Gamma^{(i)}:\mathbf{\widehat{V}}_\Gamma\rightarrow\mathbf{V}_\Gamma^{(i)}$ is the operator which maps functions in the continuous interface velocity space $\mathbf{\widehat{V}}_\Gamma$ to their subdomain components in the space $\mathbf{V}_\Gamma^{(i)}$. We denote the direct sum of the $R_\Gamma^{(i)}$ with $R_\Gamma$.\\
$\widehat{S}_\Gamma$ is assembled from the subdomain Stokes Schur complements $S_\Gamma^{(i)}$, which are defined by: given $\mathbf{w}_\Gamma^{(i)}  \in \mathbf{V}_\Gamma^{(i)}$, determine $S_\Gamma^{(i)}{\mathbf{w}}_\Gamma^{(i)}\in \mathbf{F}_\Gamma^{(i)}$ such that 
\begin{equation}\label{firstSchur}
	\left[
	\begin{array}{ccc}
		A_{II}^{(i)} & {B_{II}^{(i)}}^T & {A_{\Gamma I}^{(i)}}^T\\
		B_{II}^{(i)} & 0 & B_{I\Gamma}^{(i)}\\
		A_{\Gamma I}^{(i)} & {B_{I\Gamma}^{(i)}}^T & {A}_{\Gamma\Gamma}^{(i)}
	\end{array}
	\right]
	\left[
	\begin{array}{c}
		\mathbf{w}_I^{(i)}\\
		q_I^{(i)}\\
		\mathbf{w}_\Gamma^{(i)}
	\end{array}
	\right] = 
	\left[
	\begin{array}{c}
		\mathbf{0}\\ 0\\ S_\Gamma^{(i)}{\mathbf{w}}_\Gamma^{(i)}
	\end{array}
	\right]\text{,}
\end{equation}
these Schur complements are symmetric and positive definite (Lemma 5.1 in \cite{bevilacqua2022bddc}).
Denoting by $S_\Gamma$ the direct sum of the $S^{(i)}_\Gamma$, then $\widehat{S}_\Gamma$ is given by
\begin{equation}\label{ScapGdef}
	\widehat{S}_\Gamma = R_\Gamma^T S_\Gamma R_\Gamma = \sum_{i=1}^{N} {R^{(i)}_\Gamma}^T S^{(i)}_\Gamma R^{(i)}_\Gamma,
\end{equation} 
and then we set 
\begin{equation}\label{exampleext}
	R = \left[
	\begin{array}{cc}
		R_\Gamma  & 0\\
		0 & I
	\end{array}
	\right] \text{,} \quad 
	R^{(i)} = \left[
	\begin{array}{cc}
		R_\Gamma^{(i)}  & 0\\
		0 & I
	\end{array}
	\right]\text{.}
\end{equation}
Finally, we see from \eqref{firstSchur} that the action of $S_\Gamma^{(i)}$ on a vector can be evaluated by solving a Dirichlet problem on the subdomain $\Omega_i$ as in (\ref{uipi}),
so only the action of $\widehat{S}_\Gamma$ on a vector is required. \\
The BDDC preconditioner that we will introduce in the next Section for problem \eqref{globInt}, makes the operator of the preconditioned problem symmetric and positive definite on the so-called "benign space", so we will be able to use the preconditioned conjugate gradient (CG) method to accelerate the solution.

\subsection{Construction of the preconditioner}
Following the standard BDDC framework for the FEM in \cite{li2006bddc} and the two-dimensional one for the VEM in \cite{bevilacqua2022bddc}, we mainly need two ingredients to handle this type of algorithm. 

First, we introduce a partially assembled interface velocity space $\mathbf{\widetilde{V}}_\Gamma$:
\begin{equation}
	\mathbf{\widetilde{V}}_\Gamma = \mathbf{\widehat{V}}_\Pi \bigoplus \mathbf{V}_\Delta = \mathbf{\widehat{V}}_\Pi \bigoplus \big( \prod_{i=1}^N \mathbf{V}_\Delta^{(i)} \big).
\end{equation}
$\mathbf{\widehat{V}}_\Pi$ is the continuous coarse-level primal velocity space typically spanned by subdomain vertex nodal basis functions and/or interface edge or face basis with constant values, or with values of weight functions. We will always assume that the basis has been changed so that each primal basis function corresponds to an explicit degree of freedom. In other words, we will have explicit primal unknowns corresponding to the primal continuity constraints on edges or faces. The primal dofs are shared by neighboring subdomains. The complementary space $\mathbf{V}_\Delta$ is the direct sum of the subdomain dual interface velocity spaces $\mathbf{V}_\Delta^{(i)}$, which correspond to the remaining interface velocity dofs and are spanned by basis functions which vanish at the primal dofs. Thus, an element in the space $\mathbf{\widetilde{V}}_\Gamma$ has a continuous primal velocity and typically a discontinuous dual velocity component.

We then define an average operator that has to restore the continuity across the interface after each iteration of the iterative method by defining a scaling operator $\delta^\dagger_i(x)$ as the pseudoinverse counting functions, so:
\begin{equation}\label{pseudoinv}
	\delta^\dagger_i(x):=1/card(I_x),\quad x\in \Gamma_{i},
\end{equation}
where $I_x$ is the set of subdomains' indices that have the node $x$ on their boundaries, and $card(I_x)$ is the number of these subdomains.
Then we introduce the restriction operator $\widetilde{R}_\Gamma:\mathbf{\widehat{V}}_\Gamma\rightarrow\mathbf{\widetilde{V}}_\Gamma$ and its scaled version $\widetilde{R}_{D,\Gamma}$, that it is obtained multiplying each row that corresponds to a dual dof by its scaling operator $\delta^\dagger_i$.
So, we define the average operator $E_D = \widetilde{R}\widetilde{R}_D^T$, which maps $\mathbf{\widetilde{V}}_\Gamma\times Q_0$, with generally discontinuous interface velocities, to elements with continuous interface velocities in the same space. Note that the two operators $\widetilde{R}$ and $\widetilde{R}_D^T$ are simply the two previous ones also extended to the space of piecewise constant pressures, as in \eqref{exampleext}. Finally, we define an operator $\Bar{R}_\Gamma: \widetilde{\mathbf{V}}_\Gamma \rightarrow \mathbf{V}_\Gamma$ that maps a velocity function from the partially assembled space into the product one.\\
The preconditioner for solving the global saddle-point problem \eqref{globInt} is then:
\begin{equation}\label{BDDCprec}
	M^{-1}=\widetilde{R}_D^T \widetilde{S}^{-1} \widetilde{R}_D,
\end{equation}
where $\widetilde{S}$ is the Schur complement system that arises using the partially assembled velocity interface functions. The action of this preconditioner can be split as a sum of a coarse saddle point problem defined on the interface and local problems defined on each subdomain; we refer to \cite{bevilacqua2022bddc} for further details.
\begin{remark}
Analogously as in the VEM two-dimensional case and in the FEM framework, the preconditioned problem is symmetric and positive definite on the so-called "benign space"  $\widehat{\mathbf{V}}_{\Gamma,B} \times Q_0$ and $\widetilde{\mathbf{V}}_{\Gamma,B} \times Q_0$, where:
	\begin{align*}
		\begin{split}
			\widehat{\mathbf{V}}_{\Gamma,B} = \{\mathbf{v}_\Gamma \in \widehat{\mathbf{V}}_\Gamma | \widehat{B}_{0\Gamma}\mathbf{v}_\Gamma = 0\},\\
			\widetilde{\mathbf{V}}_{\Gamma,B} = \{\mathbf{v}_\Gamma \in \widetilde{\mathbf{V}}_\Gamma | \widetilde{B}_{0\Gamma}\mathbf{v}_\Gamma = 0\}.
		\end{split}
	\end{align*}
	This is a crucial observation, and, to ensure that the iterates of the preconditioned iterative method remain in this subspace, it is necessary that a no-net-flux condition (Assumption \ref{ass1}) holds:
\end{remark}

\begin{ass}\label{ass1}
	For any $\mathbf{v}_\Delta\in\mathbf{V}_\Delta$, $\int_{\partial\Omega_i}\mathbf{v}_\Delta^{(i)}\cdot\mathbf{n} = 0$ and $\int_{\partial\Omega_i}(E_{D}\mathbf{v}_\Delta)^{(i)}\cdot\mathbf{n} = 0$, where $\mathbf{n}$ is the outward normal of $\partial\Omega_i$. We can equivalently write $B_{0\Delta}^{(i)}\mathbf{v}_\Delta^{(i)} = 0$ and $B_{0\Delta}^{(i)}(E_{D}\mathbf{v}_\Delta)^{(i)} = 0$.
\end{ass}

\section{Convergence rate estimate}\label{sec:convrate}
First of all we introduce the $|.|_{S_\Gamma^{(i)}}$ and $|.|_{S_\Gamma}$ seminorms defined by
\begin{align}
	|\mathbf{v}_\Gamma^{(i)}|_{S_\Gamma^{(i)}}^2 = {\mathbf{v}_\Gamma^{(i)}}^TS_\Gamma^{(i)} \mathbf{v}_\Gamma^{(i)}\text{,} \quad |\mathbf{v}_\Gamma|_{S_\Gamma}^2 = {\mathbf{v}_\Gamma}^T S_\Gamma \mathbf{v}_\Gamma = \sum_{i=1}^{N} |\mathbf{v}_\Gamma^{(i)}|_{S_\Gamma^{(i)}}^2,
\end{align}
and we equip the space $\mathbf{V}_\Gamma^{(i)}$ with the seminorm introduced in Lemma \ref{ineqminimizing},with consequently the seminorm $|.|_{E(\Gamma)}$ defined on the space $\mathbf{V}_\Gamma$ by $|\mathbf{v}_\Gamma|_{E(\Gamma)}^2 = \sum_{i=1}^{N} |\mathbf{v}_\Gamma^{(i)}|_{E(\Gamma_i)}^2$.\\
We introduce also a seminorm on the space $\widetilde{\mathbf{V}}_\Gamma$:
\begin{align*}
    | \mathbf{v}_\Gamma |_{\widetilde{S}_\Gamma}^2 = \mathbf{v}_\Gamma^T \Bar{R}_\Gamma^T S_\Gamma \Bar{R}_\Gamma \mathbf{v}_\Gamma = |\mathbf{v}_\Gamma|^2_{S_\Gamma}, \qquad \forall \mathbf{v}_\Gamma \in \widetilde{\mathbf{V}}_\Gamma.
\end{align*}
The following lemma, whose proof is the three-dimensional extension of Lemma 5.2 in \cite{bevilacqua2022bddc}, ensures the equivalence of the first two seminorms that we have just introduced:
\begin{lemma}\label{Bramble}
	There exist positive constant $c_1$ and $c_2$, independent of $H$, $h$ and the shape of subdomains, such that
	\begin{equation*}
		c_1{\beta_h}^2|\mathbf{v}_\Gamma|_{S_\Gamma}^2 \leq |\mathbf{v}_\Gamma|_{E(\Gamma)}^2 \leq c_2 |\mathbf{v}_\Gamma|_{S_\Gamma}^2 \quad \forall \mathbf{v}_\Gamma \in \mathbf{V}_\Gamma\text{,}
	\end{equation*}
	where $\beta_h$ is the discrete inf-sup stability constant.
\end{lemma}
We recall that the convergence rate of the preconditioned conjugate gradient method with a BDDC preconditioner is characterized by the stability of the norm of the average operator $E_D$:
\begin{ass}
    \label{ass2}
	There exists a positive constant C, which is independent of $H$, $h$, and the number of subdomains, such that
	\begin{align*}
		|\bar{R}_\Gamma (E_{D,\Gamma}\mathbf{v}_\Gamma)|_{E(\Gamma)} \leq C \bigg(1+\log\left(\frac{H}{h}\right)\bigg)|\bar{R}_\Gamma\mathbf{v}_\Gamma|_{E(\Gamma)}\text{,} \quad \forall\mathbf{v}_\Gamma \in \widetilde{\mathbf{V}}_\Gamma.
	\end{align*}
\end{ass}
Combining Assumptions \ref{ass2} and \ref{ass2} with Lemma \ref{Bramble} we can state the following theorem:
\begin{theorem}\label{teoEst}
	Let Assumptions \ref{ass1} and \ref{ass2} hold. The preconditioned operator $M^{-1}\widehat{S}$ is then symmetric positive definite
	with respect to the bilinear form $\langle\cdot,\cdot\rangle_{\widehat{S}}$ on the benign space $\widehat{\mathbf{V}}_{\Gamma,B}\times Q_0$. Its minimum eigenvalue is 1 and its maximum eigenvalue is bounded by
	\begin{align}
		C \frac{1}{\beta_h^2}\bigg(1+\log\left(\frac{H}{h}\right)\bigg)^2\text{.}
	\end{align}
	Here, $C$ is a constant independent of $H$, $h$, and the number of subdomains, and $\beta_h$ is the discrete inf-sup stability constant.
\end{theorem}
The proof can be found in \cite{bevilacqua2022bddc}, where an argument independent of the dimension of the space is provided.

\subsection{Satisfying Assumptions and other aspects}\label{assumptions}
We now provide a recipe to construct the coarse space in three dimensions. Following \cite{li2006bddc}, we recall that in three dimensions, the interface $\Gamma$ of a subdomain $\Omega_i$ is constituted by faces $\mathcal{F}_l$ shared by two subdomains, edges $\mathcal{E}_k$ that are shared by more than two subdomains (the notation $\mathcal{E}_k (\mathcal{F}_l)$ is to underline that the edge $\mathcal{E}_k$ belongs to the face $\mathcal{F}_l$) and vertices $V_j$ that are the endpoints of the edges. Now let be $G$ any one of these geometrical entities and let be $\mathbf{v} \in \mathbf{V}$ a generic virtual function, we define the cut-off linear functional $\theta_{G}$ that maps a virtual function $\mathbf{v}$ into another virtual function $\theta_{G}(\mathbf{v})$, that is equal to $\mathbf{v}$ on all the dofs that belongs to $G$ and $0$ elsewhere, for simplicity we generally write $\theta_{G}\mathbf{v}$ instead of $\theta_{G}(\mathbf{v})$. We also claim that when we use the subscript $\mathcal{F}_l$, we make reference to the dofs that live only in the interior of the face, so we exclude the boundary face dofs. When a multi-subscript is present, like $\mathcal{F}_{ij}$, it means that the face is shared between the subdomains $\Omega_i$ and $\Omega_j$. \\
To satisfy Assumption \ref{ass1}, we first make all vertices primal, and then we require that, for any $\mathbf{v}_\Delta$, the two quantities:
\begin{align}\label{net}
	\int_{\mathcal{F}_{ij}} \mathbf{v}_\Delta^{(i)} \cdot\mathbf{n}_{ij} = \int_{\mathcal{F}_{ij}} (\theta_{{\mathcal{F}}_{ij}}\mathbf{v}_\Delta^{(i)})\cdot\mathbf{n}_{ij} + \sum_{\mathcal{E}_k \subset \mathcal{F}_{ij}} \int_{\mathcal{F}_{ij}} (\theta_{{\mathcal{E}_k}(\mathcal{F}_{ij})}\mathbf{v}_\Delta^{(i)})\cdot\mathbf{n}_{ij}
\end{align}
and
\begin{align}\label{netavg}
    \begin{split}
        \int_{\mathcal{F}_{ij}} (E_{D}\mathbf{v})_\Delta^{(i)} \cdot\mathbf{n}_{ij} = \frac{1}{2} \int_{\mathcal{F}_{ij}} \theta_{{\mathcal{F}}_{ij}} (\mathbf{v}_\Delta^{(i)}+\mathbf{v}_\Delta^{(j)})\cdot\mathbf{n}_{ij} \\ + \sum_{\mathcal{E}_k \subset \mathcal{F}_{ij}} \sum_{m \in \mathcal{N}_{\mathcal{E}_k}} \frac{1}{card(\mathcal{N}_{\mathcal{E}_k})}\int_{\mathcal{F}_{ij}} (\theta_{{\mathcal{E}_k}(\mathcal{F}_{ij})}\mathbf{v}_\Delta^{(m)})\cdot\mathbf{n}_{ij}
    \end{split}
\end{align}
vanish, where $\mathcal{N}_{\mathcal{E}_k}$ is the set of all the subdomains that share the edge $\mathcal{E}_k$ and $\mathbf{n}_{ij}$ is the unit outward normal vector to the face $\mathcal{F}_{ij}$. To do so, we need that all the integrals of the right-hand side of (\ref{net}) and (\ref{netavg}) will vanish. This can be achieved by enforcing a primal constraint for each face $\mathcal{F}_{ij}$:
\begin{align}\label{primalfacen}
    \int_{\mathcal{F}_{ij}} (\theta_{{\mathcal{F}}_{ij}}\mathbf{v}_\Gamma^{(i)})\cdot\mathbf{n}_{ij} = \int_{\mathcal{F}_{ij}} (\theta_{{\mathcal{F}}_{ij}}\mathbf{v}_\Gamma^{(j)})\cdot\mathbf{n}_{ij}
\end{align}
and a set of primal constraints requiring that for each edge $\mathcal{E}_k$, on each face $\mathcal{F}_{ij}$, the following quantity is the same for all $m\in \mathcal{N}_{\mathcal{E}_k}$:
\begin{align}\label{primaledgen}
    \int_{\mathcal{F}_{ij}}(\theta_{{\mathcal{E}_k}(\mathcal{F}_{ij})}\mathbf{v}_\Gamma^{(m)})\cdot\mathbf{n}_{ij}
\end{align}
To ensure constraint \eqref{primalfacen}, we need one primal variable per face, while, to ensure constraint \eqref{primaledgen}, we need as many primal variables as the number of faces which share the edge $\mathcal{E}_k$. 
We remark that in our VEM context, the quantities in \eqref{primalfacen} and \eqref{primaledgen} are directly computable from the dofs introduced in Section \ref{sec:2}. For particular subdomain partitions, such as those with cubic subdomains and hexahedral elements, it might happen that some of the primal basis functions are linearly dependent; this situation is harmless in practice since we can perform a singular value decomposition of the basis dofs and obtain non-singular coarse operators.

To satisfy Assumption \ref{ass2}, we have to ensure that we have the right type of constraints that can control the rigid body modes (at least six constraints: the three translations and the three rotations).
Given the fact that the coefficients of the Stokes problem are all the same for each subdomain and that the vertices of the subdomains have been selected as primal constraints, we can prove that the second Assumption is satisfied if also all the faces of the interface $\Gamma$ are \textit{fully primal} in the sense of the following definition (see \cite[def. 5.3]{klawonn2006}):
\begin{definition}\label{fullyprimaldef}
A face $\mathcal{F}_{ij}$ is called fully primal if, in the space of
primal constraints over $\mathcal{F}_{ij}$, there exists a set $f_m,\ m =1,...,6$, of linear functionals on $\mathbf{V}_\Gamma^{(i)}$ with the following properties:
    \begin{itemize}
        \item $|f_m(\mathbf{v}_\Gamma^{(i)})|^2 \leq C H^{-1} (1+\log(H/h))\Vert \mathbf{v}_\Gamma^{(i)} \Vert_{H^{1/2}(\mathcal{F}_{ij})}$;
        \item $f_m(\mathbf{r}_l) = \delta_{ml} \quad \forall m,l=1,...,6 \quad \mathbf{r}_l\in ker(\varepsilon)$,
    \end{itemize}
    with $C>0$ and $\mathbf{v}_\Gamma^{(i)}\in\mathbf{V}_\Gamma^{(i)}$.
\end{definition}
We recall that, to satisfy Assumption \ref{ass1}, we have chosen as primal constraints some averages of the normal component of the velocity over the edges. In Section 7 of \cite{li2006bddc} (see also further details in \cite{klawonn2006}), it is shown that this choice of primal dofs is sufficient to guarantee a set of functionals that, if they all vanish for an arbitrary rigid body mode, then the rigid body mode must vanish. It is essential to underline that in some particular cases, like triangular or rectangular faces, it is necessary to introduce some extra edge average in the tangential direction.
This condition can be verified numerically because the selection of a set of linearly independent set of constraints can be computed using a QR factorization and selecting six functionals that are robustly independent.\\
Before proving the stability of the average operator $E_D$ we need a lemma that is the vectorial extension of the one in \cite{bertoluzza2020}:
\begin{lemma}\label{lemmaface1}
Let be $\Omega_i$ a subdomain and let be $F$ a face of $\Omega_i$. Then for $\mathbf{v}_\Gamma \in \mathbf{V}_{\Gamma|F}$ we have 
\begin{align}
    \Vert \mathbf{v}_\Gamma \Vert_{[L^2(\partial F)]^3} \lesssim \sqrt{1+\log(H/h)} \Vert \mathbf{v}_\Gamma \Vert_{[H^{1/2}(F)]^3}
\end{align}
\end{lemma}
and a face lemma:
\begin{lemma}\label{lemmafece2}
    Let $\mathbf{v}_\Gamma \in \mathbf{V}_\Gamma^{(i)}$. Then, for all faces $F$ of $\Omega_i$ it holds that $\theta_{\mathcal{F}_i}\mathbf{v}_{\Gamma|F} \in [H^{1/2}_{00}(F)]^3$ and
    \begin{align}
        \Vert \theta_{\mathcal{F}_i}\mathbf{v}_\Gamma \Vert^2_{[H^{1/2}_{00}(F)]^3} \lesssim (1+\log(H/h))^2 \Vert \mathbf{v}_\Gamma \Vert^2_{[H^{1/2}(F)]^3}.
    \end{align}
\end{lemma}

This proof can be found again in \cite{bertoluzza2020}, where we need to use an optimal estimate for an interpolant of the VEM functions for a Stokes  problem (Lemma \ref{interpST}) and a Riesz Basis Property (Lemma \ref{rieszbp}).
\begin{lemma}\label{interpST}
    Let $\mathbf{v} \in [H^{1+s}(K)]^3$, $0 \leq s \leq 1$, there exist $\mathbf{v}_I \in \mathbf{V}_h(K)$ s.t.:
    \begin{align*}
        \Vert \mathbf{v} -\mathbf{v}_I \Vert_{0,K} + |\mathbf{v} -\mathbf{v}_I |_{1,K} \leq h^{1+s} |\mathbf{v}|_{1+s,\widetilde{K}}.
    \end{align*}
\end{lemma}
\begin{proof}
    The proof of this lemma is divided into three steps.\\
    \textbf{Step 1}. Interpolant on faces.\\
    Let be $K$ an element of the VEM tassellation and $f$ a face with $f \in \partial K$.
    We consider $\widetilde{\mathcal{T}}_h$ a sub-triangularization of $ \mathcal{T}_h$ and let be $\mathbf{v}_c$ the Clement interpolant of $\mathbf{v}$ relative to the sub-triangularization.\\
    We have that: $\Vert \mathbf{v} -\mathbf{v}_c \Vert_0 + |\mathbf{v} -\mathbf{v}_c |_1 \leq h |\mathbf{v}|_{1,f}$. Now we interpolate $\mathbf{v}_c$ on the larger face space:
    \begin{equation}\label{faceSpaceBig}
        \begin{split}
        \mathbb{B}_k(f) :=\{ v \in & H^1(f) \text{ s.t. (i) } v_{|e}\in C^0(\partial f), v_{|e}\in \mathbb{P}_k(e) \text{ for all } e \in \partial f,\\
        & \text{(ii) } \Delta_f v \in \mathbb{P}_{k+1}(f) \}
        \end{split}
    \end{equation}
    and we define $\mathbf{w}_I \in [\mathbb{B}_k(f)]^3$ as the solution of:
    \begin{equation}
        \begin{cases}
		-\Delta\mathbf{w}_I = -\Delta \Pi_{k+1}^0\mathbf{v}_c \qquad &\text{ on } f \\
		\mathbf{w}_I = \mathbf{w}_c \qquad &\text{ on } \partial f.
	\end{cases}
    \end{equation}
    We see that $\Delta \Pi_{k+1}^0\mathbf{v}_c \in [\mathbb{P}_{k-2}(f)]^3 \subset [\mathbb{P}_{k+1}(f)]^3$, $\mathbf{v}_c$ is continuous by definition on $\Omega$ and $\mathbf{v}_c|_e \in [\mathbb{P}_k(e]^3)$  $\forall e \in \partial f$, so we conclude $\mathbf{w}_I \in [\mathbb{B}_k(f)]^3$.
    Subtracting $\Pi_{k+1}^0\mathbf{v}_c$ at the second equation:
    \begin{equation}
        \begin{cases}
		-\Delta\mathbf{w}_I = -\Delta \Pi_{k+1}^0\mathbf{v}_c \qquad &\text{ on } f \\
		\mathbf{w}_I - \Pi_{k+1}^0\mathbf{v}_c = \mathbf{w}_c - \Pi_{k+1}^0\mathbf{v}_c \qquad &\text{ on } \partial f,
	\end{cases}
    \end{equation}
    we have:
    \begin{align}\label{wipvc}
        \begin{split}
            |\mathbf{w}_I-\Pi_{k+1}^0\mathbf{v}_c|_{1,f}\leq \text{inf}\{ |\mathbf{z}|_{1,f}, \mathbf{z} \in [H^1(f)]^3: \mathbf{z} = \mathbf{v}_c - \Pi_{k+1}^0\mathbf{v}_c \text{ on } \partial f \} \\ \leq |\mathbf{v}_c-\Pi_{k+1}^0\mathbf{v}_c|_{1,f}.
        \end{split}
    \end{align}
    Now using the triangular inequality:
    \begin{align}\label{vcwi}
        |\mathbf{v}_c-\mathbf{w}_I|_{1,f} \leq |\mathbf{v}_c-\Pi_{k+1}^0\mathbf{v}_c|_{1,f} + |\Pi_{k+1}^0\mathbf{v}_c- \mathbf{w}_I|_{1,f} \leq 2 |\mathbf{v}_c-\Pi_{k+1}^0\mathbf{v}_c|_{1,f}.
    \end{align}
    We have $\mathbf{w}_I \in [\mathbb{B}_k(f)]^3$, but we are looking for $\mathbf{v}_I \in [\widehat{\mathbb{B}}_k(f)]^3$, interpolating by the definition we have $\mathbf{w}_I = \mathbf{v}_I$ on $\partial f$ and:
    \begin{align}\label{moments}
        \begin{split}
            \int_f \mathbf{v}_I \cdot \mathbf{p} = \int_f \mathbf{w}_I \cdot \mathbf{p} \qquad \forall \mathbf{p} \in [\mathbb{P}_{k-2}]^3 \\
            \int_f \mathbf{v}_I \cdot \mathbf{p} = \int_f \Pi_k^{\nabla,f}\mathbf{w}_I \cdot \mathbf{p} \qquad \forall \mathbf{p} \in [\mathbb{P}_{k+1\setminus k-1}]^3,
        \end{split}
    \end{align}
    Remembering that $\mathbf{w}_I = \mathbf{v}_I$ on $\partial f$ and integrating by parts:
    \begin{align}\label{intpart}
        |\mathbf{w}_I-\mathbf{v}_I|_{1,f}^2 = \int_f |\nabla(\mathbf{w}_I-\mathbf{v}_I)|^2 = -\int_f \Delta (\mathbf{w}_I-\mathbf{v}_I)(\mathbf{w}_I-\mathbf{v}_I).
    \end{align}
    By definition $\mathbf{w}_I-\mathbf{v}_I \in [\mathbb{B}_k(f)]^3$, so $\Delta (\mathbf{w}_I-\mathbf{v}_I) \in \mathbb{P}_{k+1}(f) \implies \Delta (\mathbf{w}_I-\mathbf{v}_I) = \mathbf{p} + \Pi_{k-2}^{0,f} \Delta (\mathbf{w}_I-\mathbf{v}_I)$ for some $\mathbf{p} \in [\mathbb{P}_{k+1\setminus k-2}(f)]^3$. In this way we can write $\mathbf{p} = -\Delta (\mathbf{w}_I-\mathbf{v}_I) + \Pi_{k-2}^{0,f} \Delta (\mathbf{w}_I-\mathbf{v}_I) = -(I-\Pi_{k-2}^{0,f})\Delta (\mathbf{w}_I-\mathbf{v}_I)$.\\
    Now using \eqref{intpart}, the equivalence for the moments up to degree $k-2$ in \eqref{moments} and an inverse estimate in \cite{bertoluzza2020}:
    \begin{align}
        \begin{split}
            |\mathbf{w}_I-\mathbf{v}_I|_{1,f}^2 = -\int_f \mathbf{p} \cdot (\mathbf{w}_I-\mathbf{v}_I) - \int_f (\mathbf{w}_I-\mathbf{v}_I)\cdot \Pi_{k-2}^{0,f}\Delta(\mathbf{w}_I-\mathbf{v}_I) = \\
            -\int_f \mathbf{p} \cdot (\mathbf{w}_I-\mathbf{v}_I) = -\int_f \mathbf{p} \cdot (\mathbf{w}_I-\Pi_{k}^{\nabla,f}\mathbf{w}_I) \leq \Vert \mathbf{p} \Vert_{0,f} \cdot \Vert \mathbf{w}_I-\Pi_{k}^{\nabla,f}\mathbf{w}_I \Vert_{0,f} = \\ \Vert I-\Pi_{k-2}^{0,f} \Vert \Vert \Delta (\mathbf{w}_I-\mathbf{v}_I) \Vert_{0,f} \Vert \mathbf{w}_I-\Pi_{k}^{\nabla,f}\mathbf{w}_I \Vert_{0,f} \\ \leq c\cdot h^{-1} |\mathbf{w}_I-\mathbf{v}_I|_{1,f}\cdot  \Vert \mathbf{w}_I-\Pi_{k}^{\nabla,f}\mathbf{w}_I \Vert_{0,f}
        \end{split}
    \end{align}
    Now using a triangular inequality, the fact that $\Pi_{k}^{0,f}\mathbf{w}_I = \Pi_{k}^{\nabla,f}(\Pi_{k}^{0,f}\mathbf{w}_I)$ and a Poincare estimate:
    \begin{align}
        \begin{split}
            |\mathbf{w}_I-\mathbf{v}_I|_{1,f} = \leq c\cdot h^{-1} \Vert \mathbf{w}_I-\Pi_{k}^{\nabla,f}\mathbf{w}_I \Vert_{0,f} \leq c\cdot h^{-1} (\Vert \mathbf{w}_I-\Pi_{k}^{0,f}\mathbf{w}_I \Vert_{0,f} \\ + \Vert \Pi_{k}^{0,f}\mathbf{w}_I-\Pi_{k}^{\nabla,f}\mathbf{w}_I \Vert_{0,f}) \leq c(1+c_\Delta)\cdot h^{-1} \Vert \mathbf{w}_I-\Pi_{k}^{0,f}\mathbf{w}_I \Vert_{0,f} \leq \\
            c_2 | \mathbf{w}_I-\Pi_{k}^{0,f}\mathbf{w}_I |_{1,f}.
        \end{split}
    \end{align}
    Using a triangular inequality, by the stability of the projection $\Pi_k^{0,f}$, \eqref{vcwi} and \eqref{wipvc}:
    \begin{align}
        \begin{split}
            |\mathbf{w}_I-\mathbf{v}_I|_{1,f}\leq c_2(|\mathbf{w}_I-\Pi_k^{0,f}\mathbf{v}_c |_{1,f} + | \Pi_k^{0,f}\mathbf{v}_c-\Pi_{k}^{0,f}\mathbf{w}_I |_{1,f}) \\ \leq  c_2(1+c_0)|\mathbf{v}_c-\Pi_k^{0,f}\mathbf{v}_c |_{1,f}.
        \end{split}
    \end{align}
    We can also estimate:
    \begin{align}
        |\mathbf{v}_c-\mathbf{v}_I|_{1,f} \leq |\mathbf{v}_c-\mathbf{w}_I|_{1,f} + |\mathbf{w}_I-\mathbf{v}_I|_{1,f} \leq c |\mathbf{v}_c-\Pi_k^{0,f}\mathbf{v}_c |_{1,f}
    \end{align}
    and
    \begin{align}
        \begin{split}
            |\mathbf{v}_c-\Pi_k^{0,f}\mathbf{v}_c |_{1,f} \leq |\mathbf{v}_c-\mathbf{v}|_{1,f} + |\mathbf{v}-\Pi_k^{0,f}\mathbf{v} |_{1,f} + |\Pi_k^{0,f}\mathbf{v}-\Pi_k^{0,f}\mathbf{v}_c |_{1,f}\\ \leq 
            (1+c_0)|\mathbf{v}_c-\mathbf{v}|_{1,f} + |\mathbf{v}|_{1,f} \leq c_3 |\mathbf{v}|_{1,f}.
        \end{split}
    \end{align}
    Combining the estimates obtained and by triangular inequality we have:
    \begin{align}
        |\mathbf{v}-\mathbf{v}_I|_{1,f} \leq |\mathbf{v}-\mathbf{v}_c|_{1,f} + |\mathbf{v}_c-\mathbf{v}_I|_{1,f} \leq c |\mathbf{v}|_{1,f}
    \end{align}
    Remembering that $\mathbf{v}_I-\mathbf{v}_c|_{\partial f} = 0$ and using a \text{Poincaré-Friedrichs} type inequality we can obtain an $L^2$ estimate and gain a power $h$.\\
    \textbf{Step 2}. Interpolant on enriched element space.\\
    $\forall f \in \partial K$ we consider $\mathbf{w}_I^f$ and $\mathbf{v}_I^f$ and let be $\mathbf{w}_I^{\partial K}$ and $\mathbf{v}_I^{\partial K}$ their gluing (continuous by construction). 
    Again we consider the 3D Clement interpolant of $\mathbf{v}$ relative to the sub-tassellation $\widehat{\mathcal{T}}$ made by tetrahedra of $\mathcal{T}$ and let be $\mathbf{v}_\pi := \Pi_k^{0,K}  \mathbf{v}_c$.
    Now $\Delta \mathbf{v}_\pi \in [\mathbb{P}_k(K)]^3 \implies \Delta \mathbf{v}_\pi = \nabla q_\pi + \widehat{\mathbf{g}}$ with $q_\pi \in \mathbb{P}_{k+1}(K)$ and $\widehat{\mathbf{g}} \in \mathbf{x} \land [\mathbb{P}_{k-1}(K)]^3$.\\
    We now define $\mathbf{w}_I \in \widetilde{\mathbf{V}_h}$ as the solution of:
    \begin{equation}\label{solwI}
	\begin{cases}
		-\Delta\mathbf{w}_I - \nabla s = \widehat{\mathbf{g}} \qquad &\text{ in }K \\
		\text{div }\mathbf{w}_I = \Pi_{k-1}^{0,K}(\text{div }\mathbf{v}_c)  \qquad &\text{ in }K \\
		\mathbf{w}_I=\mathbf{v}_I^{\partial K} \qquad  &\text{ on } \partial K,
	\end{cases}
    \end{equation}
    and $\widetilde{\mathbf{v}}$ the solution of the auxiliar problem:
    \begin{equation}\label{solvtilde}
	\begin{cases}
		-\Delta \widetilde{\mathbf{v}} - \nabla \widetilde{s} = \widehat{\mathbf{g}} \qquad &\text{ in }K \\
		\text{div }\widetilde{\mathbf{v}} = \text{div }\mathbf{v}_c  \qquad &\text{ in }K \\
		\widetilde{\mathbf{v}}=\mathbf{v}_I^{\partial K} \qquad  &\text{ on } \partial K,
	\end{cases}
    \end{equation}
    Adding and subtracting $\Delta \mathbf{v}_\pi$ at the first equation of \eqref{solvtilde} and $\mathbf{v}_\pi$ at the other two:
    \begin{equation}\label{solvtilde2}
	\begin{cases}
		-\Delta (\mathbf{v}_\pi - \widetilde{\mathbf{v}}) - \nabla (- q_\pi - \widetilde{s}) = 0 \qquad &\text{ in }K \\
		\text{div }(\mathbf{v}_\pi - \widetilde{\mathbf{v}}) = \text{div }(\mathbf{v}_\pi - \mathbf{v}_c)  \qquad &\text{ in }K \\
		\widetilde{\mathbf{v}} - \mathbf{v}_\pi =\mathbf{v}_I^{\partial K} - \mathbf{v}_\pi = (\mathbf{v}_I^{\partial K} - \mathbf{v}_c)|_{\partial K} + (\mathbf{v}_c - \Pi_k^{0,K} \mathbf{v}_c)|_{\partial K} \qquad  &\text{ on } \partial K,
	\end{cases}
    \end{equation}
    we have $\forall f \in \partial K$ $(\mathbf{v}_I^{\partial K} - \mathbf{v}_c)|_{\partial f} = \mathbf{0}$. Let be $P_f$ a regular pyramid $P_f \subset K$, by the trace theorem on $f \implies \exists \mathbf{\psi}_f \in [H^1(P_f)]^3,$ $\mathbf{\psi}_f|_{\partial P_f\setminus f} = \mathbf{0}$ and $c > 0$ s.t. $|\mathbf{\psi}_f|_{1,P_f}\leq c_f \Vert \mathbf{v}_I^f - \mathbf{v}_c \Vert_{1/2,f}$. Let be $\mathbf{\psi} := \sum_{f\in \partial P} \mathbf{\psi}_f + \mathbf{v}_c - \Pi_k^{0,E}\mathbf{v}_c$.
    Now: 
    \begin{equation}
        \begin{split}
            |\mathbf{v}_\pi - \widetilde{\mathbf{v}}|_{1,K} \leq \text{inf}\{ |\mathbf{z}|_{1,K}, \mathbf{z} \in [H^1(K)]^3: \text{div}\mathbf{z} = \text{div}\,(\mathbf{v}_\pi - \mathbf{v}_c) \text{ in } K, \\ \mathbf{z} = \mathbf{v}_I^{\partial K} - \mathbf{v}_\pi \text{ on } \partial K \} \leq |\mathbf{\psi}|_{1,K} \leq \sum_{f \in \partial K}|\mathbf{\psi}_f|_{1,P_f} + |\mathbf{v}_c - \Pi_k^{0,K}\mathbf{v}_c|_{1,K} \leq \\c_f \sum_{f \in \partial K}\Vert \mathbf{v}_I^f - \mathbf{v}_c \Vert_{1/2,f} + |\mathbf{v}_c - \Pi_k^{0,K}\mathbf{v}_c|_{1,K}
        \end{split}
    \end{equation}
    We estimate now the first term of the last inequality using the Gagliardo Niremberg estimate, the standard interpolation, the fact $\forall a,b\in \mathbb{R}$ $ab\leq h^{-1}a^2+h b^2$ and a Poincare inequality:
    \begin{align}
        \begin{split}
            \Vert \mathbf{v}_I^f - \mathbf{v}_c \Vert_{1/2,f} \leq \Vert \mathbf{v}_I^f - \mathbf{v}_c \Vert_{0,f}^2 + c_{GN} \Vert \mathbf{v}_I^f - \mathbf{v}_c \Vert_{0,f}| \mathbf{v}_I^f - \mathbf{v}_c |_{1,f}\leq \\ \Vert \mathbf{v}_I^f - \mathbf{v}_c \Vert_{0,f}^2 + c_{GN} (h^{-1}\Vert \mathbf{v}_I^f - \mathbf{v}_c \Vert_{0,f}  + h| \mathbf{v}_I^f - \mathbf{v}_c |_{1,f}) \leq \\ c h (1+h)| \mathbf{v}_I^f - \mathbf{v}_c |_{1,f} \leq ch |\mathbf{v}_c- \Pi_k^{0,P}\mathbf{v}_c|_{1,f}^2 \\ \leq ch|\mathbf{v}|_{1,f}^2 \leq c|\mathbf{v}|_{1,K}^2,
        \end{split}
    \end{align}
    and we conclude $|\mathbf{v}_\pi -\widetilde{\mathbf{v}}|_{1,K} \leq c|\mathbf{v}|_{1,K}$.
    Subtracting the system \eqref{solvtilde} at \eqref{solwI}, we have:
    \begin{equation}\label{sistdiff}
	\begin{cases}
		-\Delta (\widetilde{\mathbf{v}}-\mathbf{w}_I ) - \nabla (\widetilde{s} - s) = 0 \qquad &\text{ in }K \\
		\text{div}\,(\widetilde{\mathbf{v}} - \mathbf{w}_I) = \text{div}\,\mathbf{v}_c - \Pi_{k-1}^{0,K}\text{div}\,\mathbf{v}_c  \qquad &\text{ in }K \\
		\widetilde{\mathbf{v}} - \mathbf{v}_I = \mathbf{0} \qquad  &\text{ on } \partial K,
	\end{cases}
    \end{equation}
    By the standard theory of saddle point problem \cite{boffi2013mixed}:
    \begin{align}
        \begin{split}
            |\widetilde{\mathbf{v}} - \mathbf{w}_I|_{1,K}\leq \frac{c(\alpha)}{\beta} \Vert \text{div} \,\mathbf{v}_c - \Pi_{k-1}^{0,K}\text{div}\,\mathbf{v}_c \Vert_{0,K} \leq \frac{c(\alpha)}{\beta} (\Vert (I-\Pi_{k-1}^{0,K}) \text{div}\,\mathbf{v} \Vert_{0,K} \\+ \Vert (I-\Pi_{k-1}^{0,K})(\text{div}\,\mathbf{v} - \text{div}\,\mathbf{v}_c) \Vert_{0,K}) \leq \frac{c(\alpha)}{\beta} |\mathbf{v}|_{1,\widetilde{K}},
        \end{split}
    \end{align}
    then 
    \begin{equation}\label{vminwi}
    |\mathbf{v} - \mathbf{w}_I|_{1,K}\leq |\mathbf{v} - \mathbf{v}_\pi|_{1,K} + |\mathbf{v}_\pi - \widetilde{\mathbf{v}}|_{1,K} + |\widetilde{\mathbf{v}} - \mathbf{w}_I|_{1,K} \leq |\mathbf{v}|_{1,K}.
    \end{equation}
    It left to estimate $|\mathbf{v}_I - \mathbf{w}_I|_{1,K}$ with the norm of $|\mathbf{v}|_{1,K}$, and then we can conclude. \\
    \textbf{Step 3}. Interpolant on the VEM space.\\
    Let be $\mathbf{v}_I \in \mathbf{V}^K_h$ interpolant in sense of the Dofs, we have:
    \begin{itemize}
        \item $\mathbf{v}_I = \mathbf{w}_I$ on $\partial K$ (it means that they have the same $\mathbf{D_v1},\mathbf{D_v2} \text{ and } \mathbf{D_v3}$);
        \item $\mathbf{D_v5}$ are equals ($\mathbf{v}_I - \mathbf{w}_I = \mathbf{0}$ on $\partial K \implies \int_K \text{div }(\mathbf{v}_I = \mathbf{w}_I) = 0$);
        \item $\forall \mathbf{g}\in [\mathbb{P}_{k-3}(K)]^3 \qquad \int_K \mathbf{v}_I(\mathbf{x} \land \mathbf{g}) = \int_K \mathbf{w}_I(\mathbf{x} \land \mathbf{g})$;\\
        $\forall \mathbf{g}\in [\mathbb{P}_{k-1\setminus k-3}(K)]^3 \qquad \int_K \mathbf{v}_I(\mathbf{x} \land \mathbf{g}) = \int_K \Pi_k^(\nabla,K)\mathbf{w}_I(\mathbf{x} \land \mathbf{g})$ 
        $\implies \int_K (\mathbf{v}_I-\mathbf{w}_I)(\mathbf{x} \land \mathbf{g}) = \int_K \Pi(\Pi_k^{(\nabla,K)}\mathbf{w}_I-\mathbf{w}_I)(\mathbf{x} \land \mathbf{g}) \qquad \forall \mathbf{g}\in [\mathbb{P}_{k-1}(K)]^3$ where $\Pi$ is the $L^2$ projection on the space $[\mathbb{P}_{k-1\setminus k-3}(K)]^3$.
    \end{itemize}
    Defining $\mathbf{d}_I := \mathbf{v}_I -\mathbf{w}_I$, we have $\Delta \mathbf{d}_I + \nabla \widetilde{s} = \widehat{\mathbf{g}}$ for some $\widetilde{s} \in L^2_0(K), \widehat{\mathbf{g}}\in [\mathbb{P}_{k-1}(K)]^3$. We consider the problem:
    \begin{equation}\label{dI}
	\begin{cases}
		-\Delta\mathbf{d}_I - \nabla \widetilde{s} = \widehat{\mathbf{g}} \qquad &\text{ in }K \\
		\text{div}\,\mathbf{d}_I = 0  &\text{ in }K \\
		\mathbf{d}_I=\mathbf{0} \qquad  &\text{ on } \partial K\\
		\int_K \mathbf{d}_I(\mathbf{x} \land \mathbf{g}) = \int_K \Pi(\Pi_k^{(\nabla,K)}\mathbf{w}_I-\mathbf{w}_I)(\mathbf{x} \land \mathbf{g}) \qquad & \forall \mathbf{g}\in [\mathbb{P}_{k-1}(K)]^3,
	\end{cases}
    \end{equation}
    after providing an inf-sup condition (same technique in \cite{da2018virtual} Sec. 4), this problem is well-posed and we have the stability estimate: 
    \begin{equation}
        h|\mathbf{d}_I|_{1,K}+\Vert \widetilde{s} \Vert_{0,K} + \Vert \mathbf{g} \Vert_{0,K} \leq \Vert \Pi(\Pi_k^{(\nabla,K)}\mathbf{w}_I-\mathbf{w}_I) \Vert_{0,K},
    \end{equation} 
    and then by the stability of $\Pi$ and by the triangular inequality:
    \begin{align}
        \begin{split}\label{viminwi}
            |\mathbf{d}_I|_{1,K} \leq h^{-1} \Vert \Pi(\Pi_k^{(\nabla,K)} \mathbf{w}_I -\mathbf{w}_I) \Vert_{0,K} \leq | \Pi_k^{(\nabla,K)}\mathbf{w}_I -\mathbf{w}_I |_{1,K} \\ \leq 2| \mathbf{w}_I -\mathbf{v} |_{1,K} + | \Pi_k^{(\nabla,K)}\mathbf{v} -\mathbf{v} |_{1,K} \leq c|\mathbf{v} |_{1,K}.
        \end{split}
    \end{align} 
    We conclude by triangular inequality combining \eqref{vminwi} and \eqref{viminwi}
    The $|\cdot|_{0,K}$ estimate can be recovered again using a Poincare-Friedrichs type inequality.
\end{proof}

The Riesz basis property for our choice of the dofs for the VEM face space of functions gives us the equivalence between the $L^2(f)$ norm of a function in $\widehat{\mathbb{B}}_k(f)$ and the euclidean norm of the vector of its dofs:
\begin{lemma}\label{rieszbp}
    Let be $f$ a face of an element $K$. For all $\mathbf{v}_h \in [\widehat{\mathbb{B}}_k(f)]^3$ we have:
    \begin{align}
        \int_{f} |\mathbf{v}_h|^2 \simeq h^2 \sum_{i \in \mathcal{X}} |\mathbf{D}_\mathbf{v}^i(\mathbf{v}_h)|^2,
    \end{align}
        where $\mathcal{X}$ is the union of the set of dofs $\mathbf{D}_\mathbf{v}^1, \mathbf{D}_\mathbf{v}^2$ and $\mathbf{D}_\mathbf{v}^3$.
\end{lemma}
We do not provide the proof of this second Lemma because it is the natural vector extension of the ones in \cite{bertoluzza2020} and \cite{chen2018some}, with a slightly different choice for the dofs of the face velocity space. Moreover, it simply involves the $\Pi_{k+1}^0$ projection instead of $\Pi_{k}^0$.

We are now ready to state the lemma related to the stability of the average operator:
\begin{lemma}\label{minimalest}
For all $\mathbf{v}_\Gamma \in \widetilde{\mathbf{V}}_\Gamma$ holds that:
\begin{align}
    |E_D \mathbf{v}_\Gamma|_{\widetilde{S}}^2\lesssim C \frac{1}{\beta^2}\bigg( 1+\log\bigg(\frac{H}{h}\bigg)\bigg)^2 |\mathbf{v}_\Gamma|_{\widetilde{S}}^2,
\end{align}
where $C$ is a positive constant that is independent on $h,H$ and $\beta_h$, but it can depend on the degree $k$ of the virtual element discretization.
\end{lemma}
\begin{proof}
Let consider $\mathbf{v}_\Gamma \in \widetilde{\mathbf{V}}_\Gamma$ and we define $\mathbf{w}_\Gamma:=E_D \mathbf{v}_\Gamma$. We have:
\begin{align}
    |\mathbf{w}_\Gamma|_{\widetilde{S}} = |\mathbf{w}_\Gamma - \mathbf{v}_\Gamma + \mathbf{v}_\Gamma|_{\widetilde{S}}\leq |\mathbf{w}_\Gamma -\mathbf{v}_\Gamma|_{\widetilde{S}} + |\mathbf{v}_\Gamma|_{\widetilde{S}}.
\end{align}
Since all the vertices of the subdomains are primal, we can rewrite: 
\begin{align}
    |\mathbf{w}_\Gamma - \mathbf{v}_\Gamma|_{\widetilde{S}_\Gamma} = \sum_{i=1}^N | \mathbf{w}_\Gamma^{(i)} - \mathbf{v}_\Gamma^{(i)}|_{S_\Gamma^{(i)}} 
\end{align}
and for each subdomain $\Omega_i$ we can also use the split:
\begin{align}
    \mathbf{w}_\Gamma^{(i)} - \mathbf{v}_\Gamma^{(i)} = \sum_{\mathcal{F}_{ij} \subset \partial \Omega_i} \theta_{\mathcal{F}_{ij}}(\mathbf{w}_\Gamma^{(i)} - \mathbf{v}_\Gamma^{(i)}) + \sum_{\mathcal{E}_k \subset \partial \Omega_i} \theta_{\mathcal{E}_{k}} ( \mathbf{w}_\Gamma^{(i)} - \mathbf{v}_\Gamma^{(i)})
\end{align}
Recalling that a face $\mathcal{F}_{ij}$ is shared by two subdomains $i,j$ and using the explicit definition of $E_D$:
\begin{align*}
    \mathbf{w}_\Gamma^{(i)} - \mathbf{v}_\Gamma^{(i)} = \frac{1}{2}(\mathbf{v}_\Gamma^{(i)} + \mathbf{v}_\Gamma^{(j)}) - \mathbf{v}_\Gamma^{(i)} = \mathbf{v}_\Gamma^{(j)} - \mathbf{v}_\Gamma^{(i)}.
\end{align*}
Starting from the face contributions, we write:
\begin{align}
    \mathbf{v}_\Gamma^{(j)} - \mathbf{v}_\Gamma^{(i)} = \bigg(\mathbf{v}_\Gamma^{(j)} - \sum_{m=1}^6 f_m^{\mathcal{F}_{ij}} (\mathbf{v}_\Gamma^{(j)})\mathbf{r}_m\bigg) - \bigg(\mathbf{v}_\Gamma^{(i)} - \sum_{m=1}^6 f_m^{\mathcal{F}_{ij}} (\mathbf{v}_\Gamma^{(i)})\mathbf{r}_m\bigg)
\end{align}
where $\{\mathbf{r}_m\}$ for $m = 1,..,6$ is the basis for the rigid body modes and the $f_m^{\mathcal{F}_{ij}}(\cdot)$ are the functionals that are equal for the faces $i$ and $j$ since the faces are fully primal.
For an arbitrary rigid body mode, $\mathbf{r}^{(i)}\in \mathbf{V}_\Gamma^{(i)}$ we write:
\begin{align}\label{facerig}
    \mathbf{v}_\Gamma^{(i)} - \sum_{m=1}^6 f_m^{\mathcal{F}_{ij}} (\mathbf{v}_\Gamma^{(i)})\mathbf{r}_m = (\mathbf{v}_\Gamma^{(i)} - \mathbf{r}^{(i)}) - \sum_{m=1}^6 f_m^{\mathcal{F}_{ij}} (\mathbf{v}_\Gamma^{(i)} - \mathbf{r}^{(i)})\mathbf{r}_m
\end{align}
We can estimate the first term of the right-hand side using lemma \ref{lemmafece2}:
\begin{align}
    \Vert \theta_{\mathcal{F}_{ij}} (\mathbf{v}_\Gamma^{(i)} - \mathbf{r}^{(i)}) \Vert_{[H_{00}^{1/2}(\mathcal{F}_{ij})]^3} \lesssim (1+\log(H/h)) \Vert \mathbf{v}_\Gamma^{(i)} - \mathbf{r}^{(i)} \Vert_{[H^{1/2}(\mathcal{F}_{ij})]^3}.
\end{align}
Then we consider the second term of \eqref{facerig}, and we estimate it using two additional contributions, by lemma 8 in \cite{klawonn2006}:
\begin{align}
    \Vert \theta_{\mathcal{F}_{ij}}\mathbf{r}_m^{(i)} \Vert_{[H_{00}^{1/2}(\mathcal{F}_{ij})]^3} \lesssim H(1+\log(H/h))
\end{align}
and by the definition of fully primal face \eqref{fullyprimaldef} and \eqref{lemmaface1} we have:
\begin{align}\label{functest}
    |f_m^{\mathcal{F}_{ij}}(\mathbf{v}_\Gamma^{(i)} - \mathbf{r}^{(i)})| \lesssim
    \frac{1}{H}(1+\log(H/h))\Vert \mathbf{v}_\Gamma^{(i)} - \mathbf{r}^{(i)} \Vert_{[H^{1/2}(\mathcal{F}_{ij})]^3}.
\end{align}
Then combining the previous two estimates we have:
\begin{align}
    \Vert \theta_{\mathcal{F}_{ij}} \sum_{m=1}^6 f_m^{\mathcal{F}_{ij}} (\mathbf{v}_\Gamma^{(i)} - \mathbf{r}^{(i)})\mathbf{r}_m \Vert_{[H_{00}^{1/2}(\mathcal{F}_{ij})]^3} \lesssim (1+\log(H/h))\Vert \mathbf{v}_\Gamma^{(i)} - \mathbf{r}^{(i)} \Vert_{[H^{1/2}(\mathcal{F}_{ij})]^3}.
\end{align}
By triangular inequality and since $\mathbf{r}^{(i)}$ is an arbitrary rigid body mode, we can take the minimum of over all the modes and by Lemma \ref{ineqminimizing}:
\begin{align}
    \begin{split}
        \Vert \theta_{\mathcal{F}_{ij}}(\mathbf{v}_\Gamma^{(i)} - \sum_{m=1}^6 f_m^{\mathcal{F}_{ij}} (\mathbf{v}_\Gamma^{(i)})\mathbf{r}_m) \Vert_{[H_{00}^{1/2}(\mathcal{F}_{ij})]^3} \lesssim \Vert \theta_{\mathcal{F}_{ij}}(\mathbf{v}_\Gamma^{(i)} - \mathbf{r}^{(i)}) \Vert_{[H_{00}^{1/2}(\mathcal{F}_{ij})]^3} + \\ \Vert \theta_{\mathcal{F}_{ij}}(\mathbf{v}_\Gamma^{(i)} - \sum_{m=1}^6 f_m^{\mathcal{F}_{ij}} (\mathbf{v}_\Gamma^{(i)}- \mathbf{r}^{(i)})\mathbf{r}_m) \Vert_{[H_{00}^{1/2}(\mathcal{F}_{ij})]^3} \lesssim \\(1+\log(H/h))\Vert \mathbf{v}_\Gamma^{(i)} - \mathbf{r}^{(i)} \Vert_{[H^{1/2}(\mathcal{F}_{ij})]^3} \lesssim (1+\log(H/h))| \mathbf{v}_\Gamma^{(i)} |_{[H^{1/2}(\mathcal{F}_{ij})]^3}
    \end{split}
\end{align}
We can repeat in the same way for the $j$th term and obtain:
\begin{align}\label{faceest}
\begin{split}
    \Vert \theta_{\mathcal{F}_{ij}} (\mathbf{v}_\Gamma^{(i)} - \mathbf{v}_\Gamma^{(j)}) \Vert_{[H_{00}^{1/2}(\mathcal{F}_{ij})]^3} \lesssim (1+\log(H/h))| \mathbf{v}_\Gamma^{(i)} |_{[H^{1/2}(\mathcal{F}_{ij})]^3}\\ + (1+\log(H/h))| \mathbf{v}_\Gamma^{(j)} |_{[H^{1/2}(\mathcal{F}_{ij})]^3}
\end{split}
\end{align}
Regarding the edge terms, we need to estimate contributions that depend on the number of subdomains that share the edge. We propose the estimate for one of these contributions since the others are treated similarly.
We consider an edge $\mathcal{E}_k \subset \partial \mathcal{F}_{ij}$, by the fact that all the faces are fully primal, we can reduce these terms to face estimates, we write:
\begin{align}
    \begin{split}
        \Vert \mathbf{v}_\Gamma^{(i)} - \mathbf{v}_\Gamma^{(j)} \Vert_{[L^{2}(\mathcal{E}_{k})]^3}^2 \lesssim \Vert \mathbf{v}_\Gamma^{(i)} - \sum_{m=1}^6 f_m^{\mathcal{F}_{ij}}(\mathbf{v}_\Gamma^{(i)})\mathbf{r}_m \Vert_{[L^{2}(\mathcal{E}_{k})]^3}^2 \\+ \Vert \mathbf{v}_\Gamma^{(j)} - \sum_{m=1}^6 f_m^{\mathcal{F}_{ij}}(\mathbf{v}_\Gamma^{(j)})\mathbf{r}_m \Vert_{[L^{2}(\mathcal{E}_{k})]^3}^2
    \end{split}
\end{align}
We proceed again considering an arbitrary rigid body mode $\mathbf{r}^{(i)} \in \mathbf{V}_\Gamma^{(i)}$. Using the triangular inequality and \eqref{lemmaface1}:
\begin{align}
    \begin{split}
        \Vert \mathbf{v}_\Gamma^{(i)} - \sum_{m=1}^6 f_m^{\mathcal{F}_{ij}}(\mathbf{v}_\Gamma^{(i)})\mathbf{r}_m \Vert_{[L^{2}(\mathcal{E}_{k})]^3}^2 \lesssim \Vert \mathbf{v}_\Gamma^{(i)} - \mathbf{r}^{(i)} \Vert_{[L^{2}(\mathcal{E}_{k})]^3} + \\ \Vert \sum_{m=1}^6 f_m^{\mathcal{F}_{ij}}(\mathbf{v}_\Gamma^{(i)} - \mathbf{r}^{(i)}) \mathbf{r}_m \Vert_{[L^{2}(\mathcal{E}_{k})]^3}^2 \lesssim
        (1+\log(H/h))\Vert \mathbf{v}_\Gamma^{(i)} \Vert_{[H^{1/2}(\mathcal{F}_{ij})]^3}^2 + \\ 
        \sum_{m=1}^6 |f_m^{\mathcal{F}_{ij}}(\mathbf{v}_\Gamma^{(i)} - \mathbf{r}^{(i)})|^2 \Vert \mathbf{r}_m\Vert_{[L^{2}(\mathcal{E}_{k})]^3}^2.
    \end{split}
\end{align}
It can be proved that $\Vert \mathbf{r}_m\Vert_{[L^{2}(\mathcal{E}_{k})]^3} \lesssim H$ (\cite{klawonn2006dual}), and now using \eqref{functest} and minimizing again on all over the rigid body modes:
\begin{align}
    \begin{split}
        \Vert \mathbf{v}_\Gamma^{(i)} - \sum_{m=1}^6 f_m^{\mathcal{F}_{ij}}(\mathbf{v}_\Gamma^{(i)})\mathbf{r}_m \Vert_{[L^{2}(\mathcal{E}_{k})]^3}^2 \lesssim (1+\log(H/h))| \mathbf{v}_\Gamma^{(j)} |^2_{[H^{1/2}(\mathcal{F}_{ij})]^3}.
        \end{split}
\end{align}
We have an analogous result for the  $j$th term and obtain:
\begin{align}\label{edgeest}
    \begin{split}
        \Vert \mathbf{v}_\Gamma^{(i)} - \mathbf{v}_\Gamma^{(j)} \Vert_{[L^{2}(\mathcal{E}_{k})]^3}^2
        \lesssim (1+\log(H/h))| \mathbf{v}_\Gamma^{(i)} |^2_{[H^{1/2}(\mathcal{F}_{ij})]^3} \\+ (1+\log(H/h))| \mathbf{v}_\Gamma^{(j)} |^2_{[H^{1/2}(\mathcal{F}_{ij})]^3}.
    \end{split}
\end{align}
We conclude by Lemma \ref{Bramble}, combining \eqref{faceest} and \eqref{edgeest} by summing over the subdomains.  
\end{proof}
\section{Adaptivity}
 We recall here an adaptive technique to enrich the minimal primal space $\mathbf{V}_\Gamma$ \cite{dassiZS2022}. The idea is to solve generalized eigenvalue problems defined on each subdomain face $\mathcal{F}$ and then construct an enriched primal space such that the condition number of the preconditioned system will be bounded from above by a selected $\nu_{tol}\in [1,\infty)$ times a constant independent on $h, H$ and $N$.
 To construct an adaptive coarse space, we need to settle in a deluxe scaling context \cite{zampini2017multilevel}, so for each face $\mathcal{F}$ shared by two subdomains $i,j$, we consider the principal minors of the subdomain matrices $S^{(k)}_T$ with $k=i,j$:
 \begin{equation*}
     S_{\mathcal{F}\mathcal{F}}^{(k)} := R_\mathcal{F}^{(k)}S^{(k)}_T{R_\mathcal{F}^{(k)}}^T,
 \end{equation*}
 where $R_\mathcal{F}^{(k)}$ maps $\mathbf{V}_\Gamma^{(k}$ to the dofs located on $F$. Then we split the matrices as follows
\begin{equation*}\label{princMin}
	S_{\mathcal{F}\mathcal{F}}^{(k)} = \left[
	\begin{array}{cc}
		S_{\mathcal{F}'\mathcal{F}'}^{(k)} & S_{F'F_\Delta}^{(k)}\\
		S_{\mathcal{F}'\mathcal{F}_\Delta}^{(k)^T} & S_{\mathcal{F}_\Delta \mathcal{F}_\Delta}^{(k)}
	\end{array}
	\right], \qquad k=i,j
\end{equation*}
where $\mathcal{F}_\Delta$ is the dual set of the dofs associated to the face $\mathcal{F}$ and $\mathcal{F}':= \Gamma_i \setminus \mathcal{F}_\Delta$. We introduce the Schur complements:
\begin{equation*}
    \widetilde{S}_{\mathcal{F}_\Delta \mathcal{F}_\Delta}^{(k)} = S_{\mathcal{F}_\Delta \mathcal{F}_\Delta}^{(k)}-S_{\mathcal{F}'\mathcal{F}_\Delta}^{(k)^T} S_{\mathcal{F}'\mathcal{F}'}^{(k)^{-1}} S_{\mathcal{F}'\mathcal{F}_\Delta}^{(k)}, \quad k=i,j.
\end{equation*}
and then we solve the following eigenvalue problems:
\begin{equation*}
    \widetilde{S}_{\mathcal{F}_\Delta \mathcal{F}_\Delta}^{(i)} : \widetilde{S}_{\mathcal{F}_\Delta \mathcal{F}_\Delta}^{(j)} \psi = \nu S_{\mathcal{F}_\Delta \mathcal{F}_\Delta}^{(i)} : S_{\mathcal{F}_\Delta \mathcal{F}_\Delta}^{(j)} \psi
\end{equation*}
where $A:B = (A^{-1}+B^{-1})^{-1}$, finally we choose the element of the primal space as $S_{\mathcal{F}_\Delta \mathcal{F}_\Delta}^{(i)} : S_{\mathcal{F}_\Delta \mathcal{F}_\Delta}^{(j)} \Psi$, where $\Psi$ is the matrix formed column-wise by those eigenvectors associated with eigenvalues smaller than a fixed tolerance $1/\nu_{tol}$.\\
We do not provide a proof of the following theorem, and we remand to \cite{dassiZS2022} for further details:
 \begin{theorem}
     Let the dual space satisfy the no-net-flux condition given in \eqref{ass1} and let the average operator preserve subdomain normal fluxes as in \eqref{primalfacen} and \eqref{primaledgen}. Then, $M^{-1} S$ is symmetric positive definite on the subspace $\widehat{\mathbf{V}}_{\Gamma,B}\times Q_0$; the minimum eigenvalue is 1, and we can algebraically construct a primal space $\mathbf{V}_\Gamma$ such that:
     \begin{align}
         \kappa_2 (M^{-1}S) \leq C\nu_\text{tol} , \qquad \forall \nu_\text{tol} \in [1,\infty),
     \end{align}
     where $C$ is independent of $N , h$, and $H$ .
 \end{theorem}

\begin{figure}[tbhp]
\centering
\subfloat{\label{fig:m1}\includegraphics[width=0.33\textwidth]{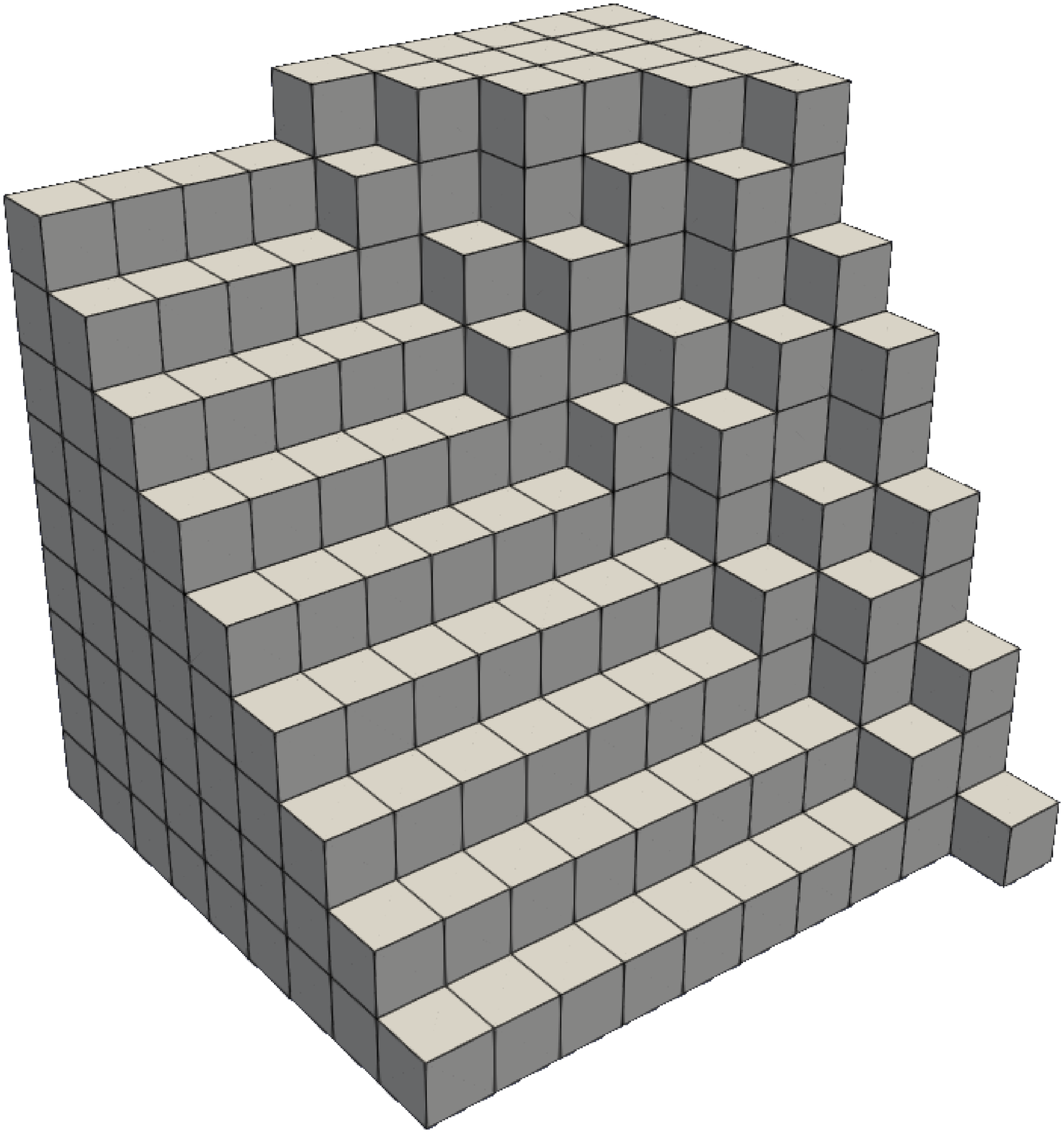}}
\subfloat{\label{fig:m2}\includegraphics[width=0.33\textwidth]{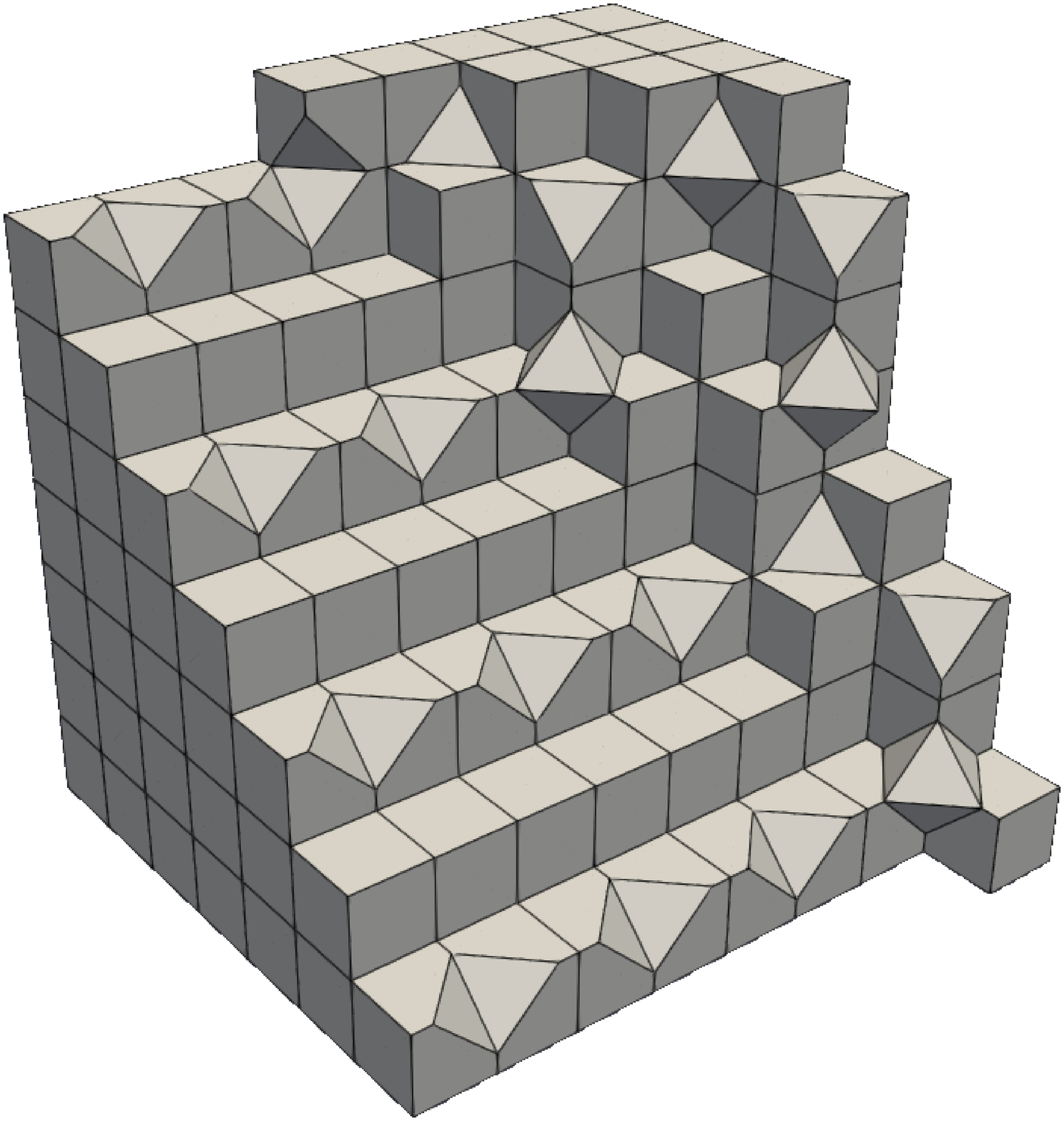}}
\subfloat{\label{fig:m3}\includegraphics[width=0.33\textwidth]{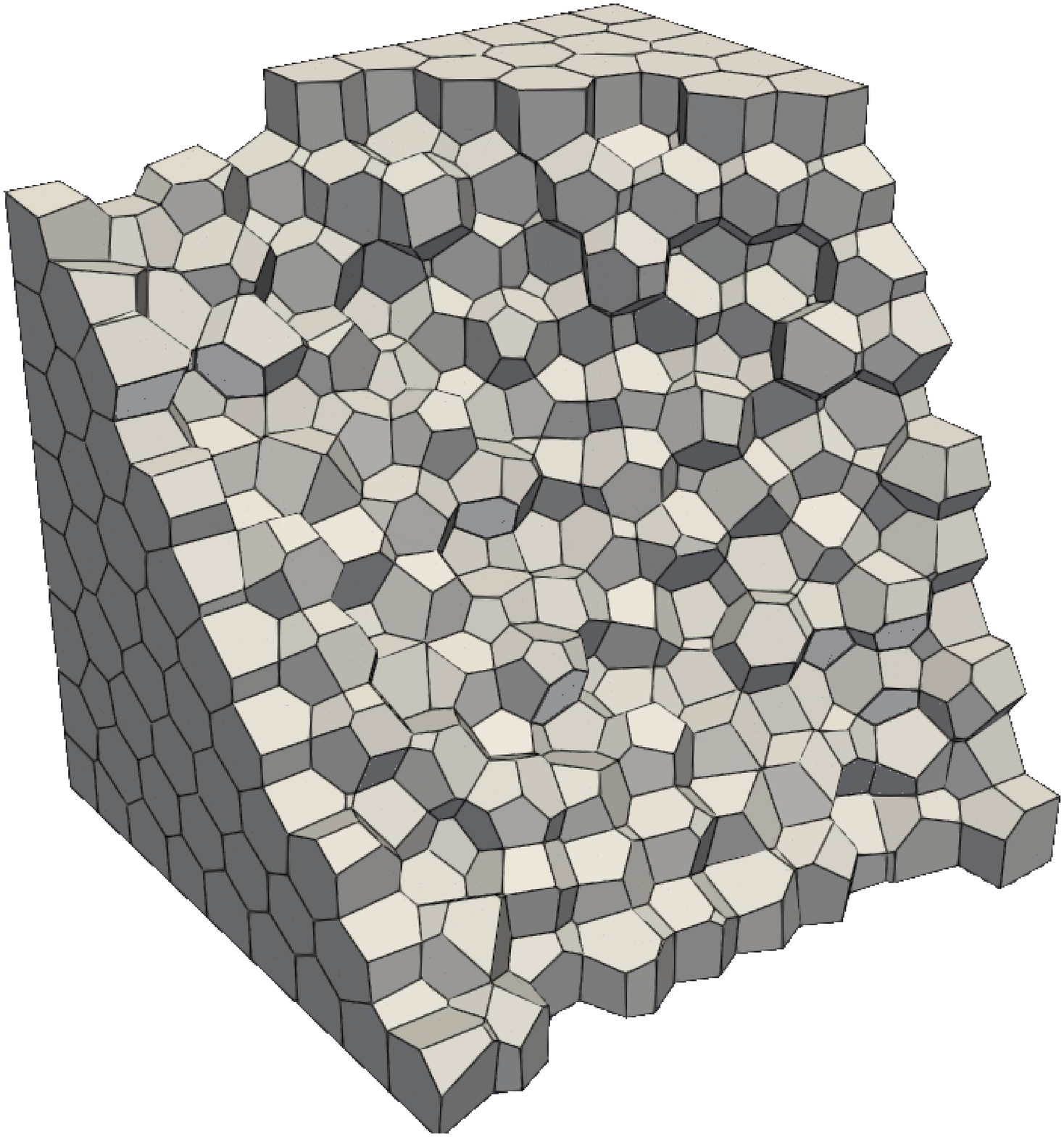}} \\
\caption{Example of CUBE, OCTA and CVT mesh discretization.}
\label{fig:mesh}
\end{figure}

\begin{figure}[tbhp]
\centering
\subfloat{\label{fig:v1}\includegraphics[width=0.5\textwidth]{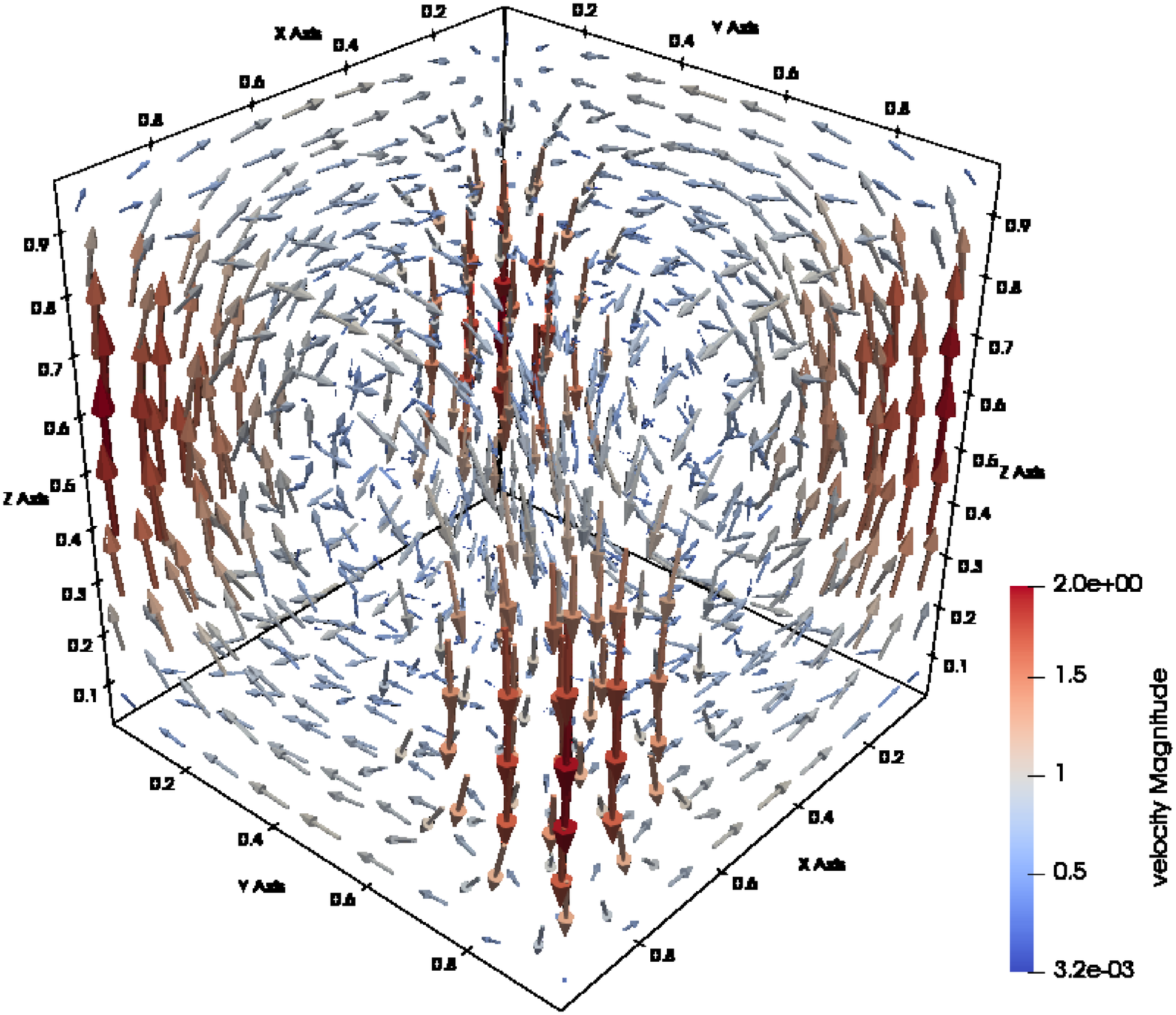}}
\subfloat{\label{fig:p1}\includegraphics[width=0.5\textwidth]{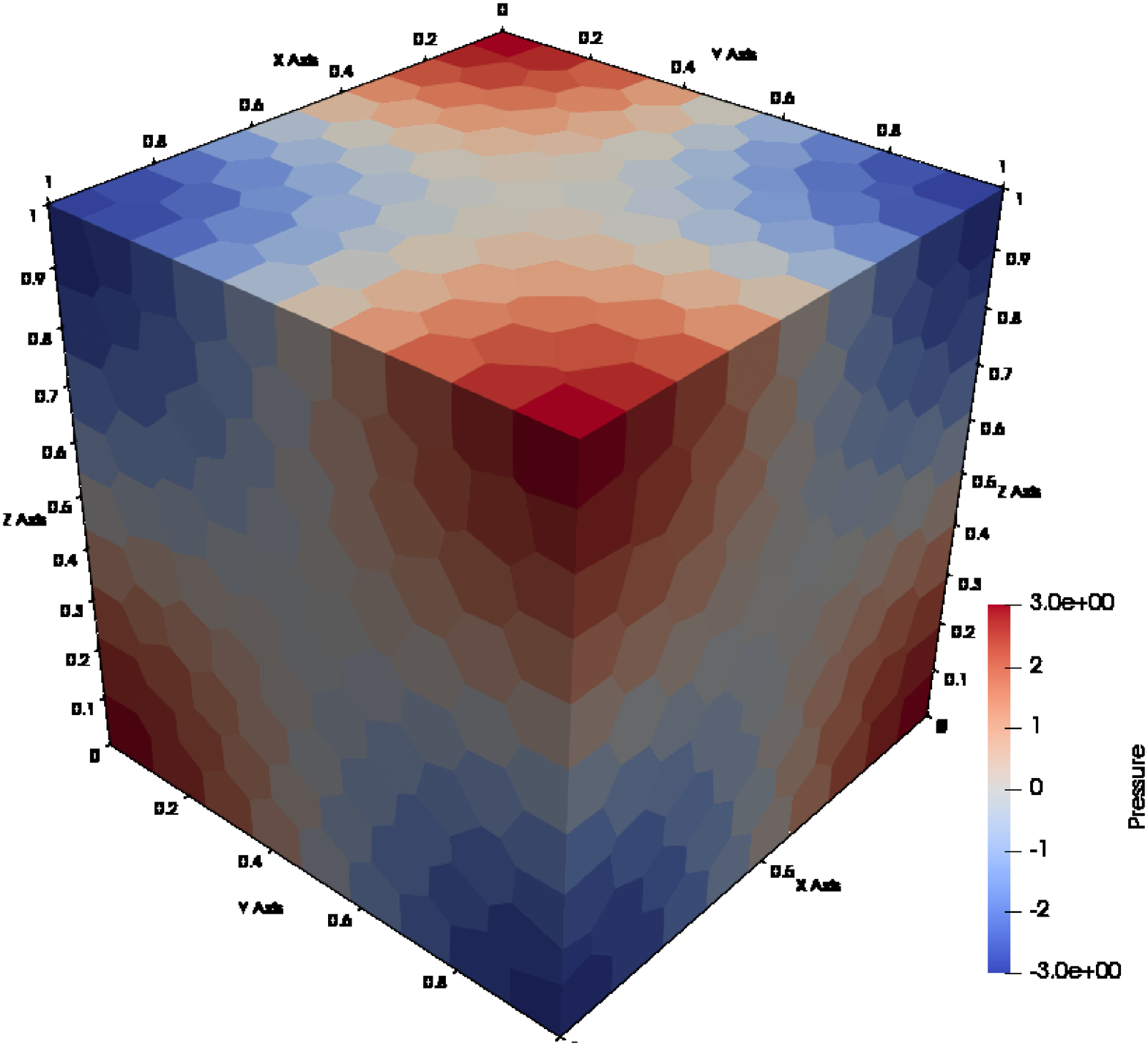}} \\
\caption{3D plot of the solution for the velocity and the pressure for our test case.}
\label{fig:solPlot}
\end{figure}

\begin{table}[tbhp]
    {\footnotesize
    \captionsetup{position=top} 
    \begin{center}
            \begin{tabular}{|r r|r|r r|r r r r|r r r r|}
			        \hline
					& & & & & \multicolumn{4}{c|}{$\nu_{tol} = 2$} & \multicolumn{4}{c|}{$\nu_{tol} = \infty$} \\
					& procs & $S^{id}$ & $T_{ass}$ & $S_p$ & $N_\Pi$ & it & $T_{sol}$ & $S_p$ & $N_\Pi$ & it & $T_{sol}$ & $S_p$ \\
                    \hline
                    \multirow{7}{*}{\STAB{\rotatebox[origin=c]{90}{CUBE}}}
					& 4   &    & \numgru{1299} &   & 447  & 8 & 129 &  & 4   & 21 & 132 & \\
					& 8   &  2  & 871 & 1.5 & \numgru{1145}  & 9 & 46 & 2.8 & 21   & 31 & 62 & 2.1\\
					& 16  & 4  & 418 & 3.1 & \numgru{2195}  & 9 & 29 & 4.4 & 41   & 33 & 35 & 3.8\\
					& 32  & 8  & 202 & 6.4 & \numgru{4577}  & 9 & 12 & 10.8 & 125  & 55 & 16 & 8.3\\
					& 64  & 16 & 105 & 12.4 & \numgru{7831}  & 9 & 5  & 25.8 & 311  & 58 & 6  & 22.0\\
					& 128 & 32 & 53  & 24.5 & \numgru{13411} & 9 & 4  & 32.3 & 643  & 61 & 5  & 26.4\\
					& 256 & 64 & 27  & 48.1 & \numgru{22723} & 9 & 5  & 25.8 & \numgru{1399} & 62 & 5  & 26.4\\
					\hline
                        \hline
                    \multirow{7}{*}{\STAB{\rotatebox[origin=c]{90}{CVT}}}
					& 4   &    & \numgru{2456} &  & 661  & 9 & 754 &  & 15 & 38 & 711 & \\
					& 8   & 2  & \numgru{1423} & 1.7 & \numgru{1655}  & 10 & 334 & 2.3 & 130   & 41 & 248 & 2.9\\
					& 16  & 4  & 816  & 3.0 & \numgru{3289}  & 10 & 157 & 4.8 & 337   & 42 & 174 & 4.1\\
					& 32  & 8  & 400  & 6.1 & \numgru{6641}  & 10 & 61  & 12.4 & 905   & 46 & 66  & 10.8\\
					& 64  & 16 & 216  & 11.4 & \numgru{10355} & 10 & 21  & 35.9 & 2326  & 33 & 21  & 33.9\\
					& 128 & 32 & 126  & 19.5 & \numgru{17566} & 11 & 11  & 68.5 & 4398  & 32 & 9   & 79.0\\
					& 256 & 64 & 72   & 34.1 & \numgru{27429} & 11 & 9   & 83.8 & \numgru{10608} & 32 & 8   & 93.2\\
					\hline
                        \hline
					  \multirow{7}{*}{\STAB{\rotatebox[origin=c]{90}{OCTA}}}
                        & 4   &    & \numgru{3225} &  & 476  & 8 & 127 &  & 21   & 32 & 59 & \\
					& 8   & 2  & \numgru{1637} & 2.0 & \numgru{1217}  & 8 & 50 & 2.6 & 21   & 32 & 59 & 2.6\\
					& 16  & 4  & 838  & 3.9 & \numgru{2299}  & 8 & 30 & 4.3 & 51   & 41 & 36 & 4.2\\
					& 32  & 8  & 434  & 7.4 & \numgru{4793} & 8 & 12 & 10.5 & 125  & 57 & 16 & 9.5\\
					& 64  & 16  & 218  & 14.8 & \numgru{8063} & 8 & 5  & 24.0 & 311  & 64 & 7  & 21.7\\
					& 128 & 32 & 116  & 27.8 & \numgru{14342} & 9 & 6  & 21.9 & 643  & 76 & 5  & 30.4\\
					& 256 & 64 & 56   & 57.6 & \numgru{26708} & 9 & 5  & 23.5 & \numgru{1676} & 70 & 7 & 21.7 \\
					\hline
				\end{tabular}
    	\end{center}}
     \vspace{1mm}
     \caption{Test 1. Strong Scaling with $k=2$. Number of elements for CUBE = $\numgru{13824}$ , CVT = $\numgru{4000}$ and OCTA = $\numgru{15552}$.}\label{strongScal}
\end{table}

\begin{figure}[!tbhp]
\centering
\subfloat{\label{fig:a}\includegraphics[width=0.49\textwidth]{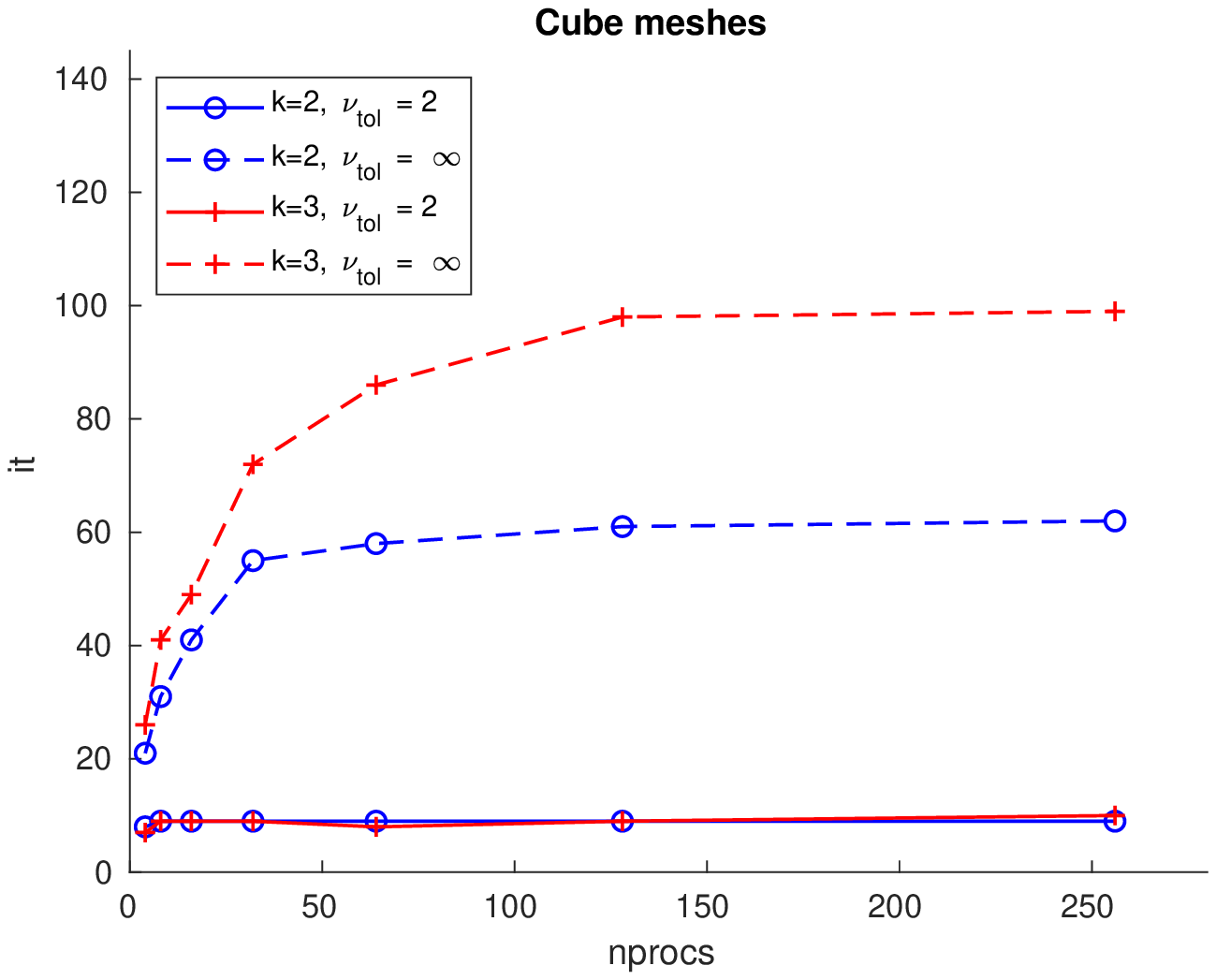}}
\subfloat{\label{fig:b}\includegraphics[width=0.49\textwidth]{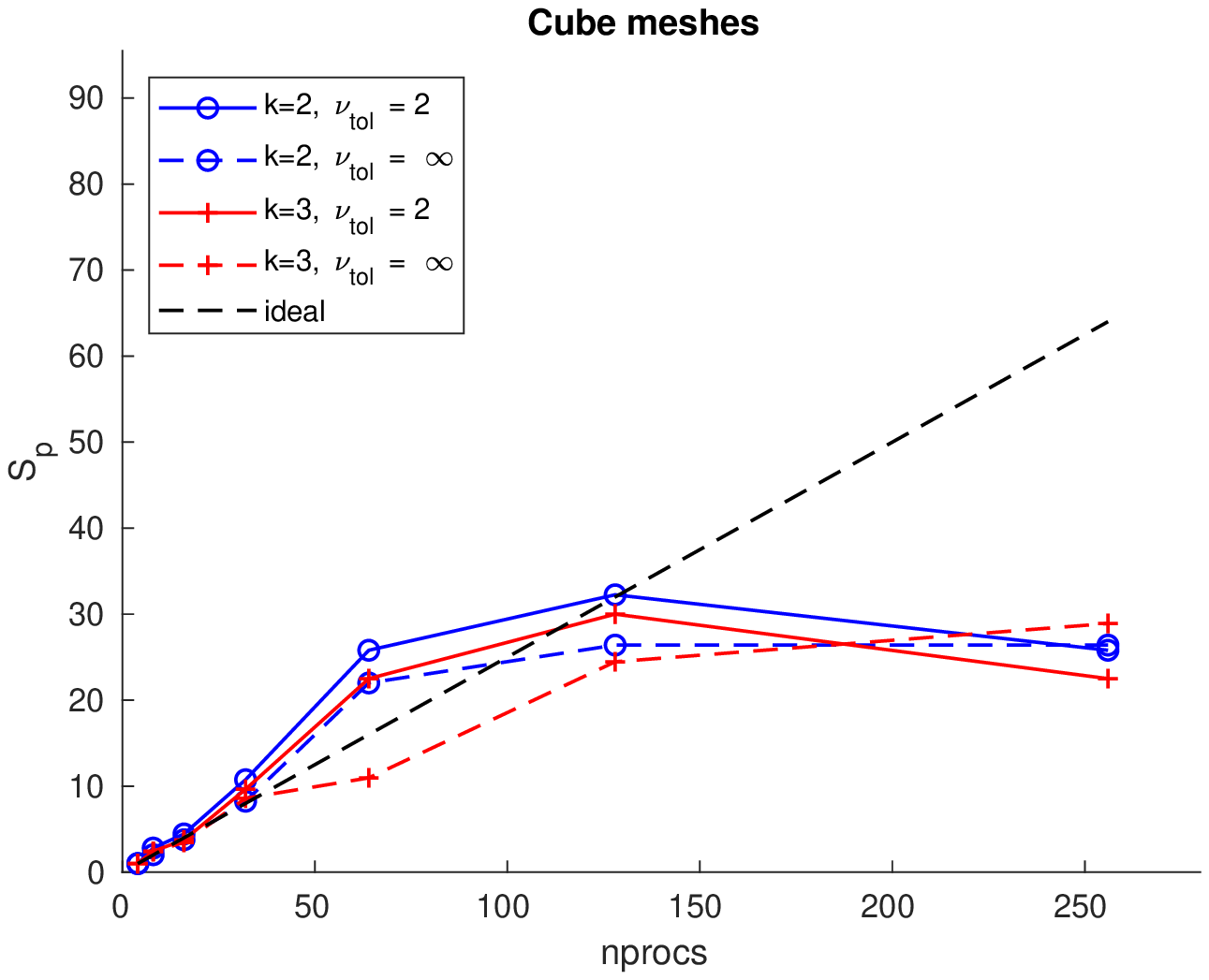}} \\
\subfloat{\label{fig:c}\includegraphics[width=0.49\textwidth]{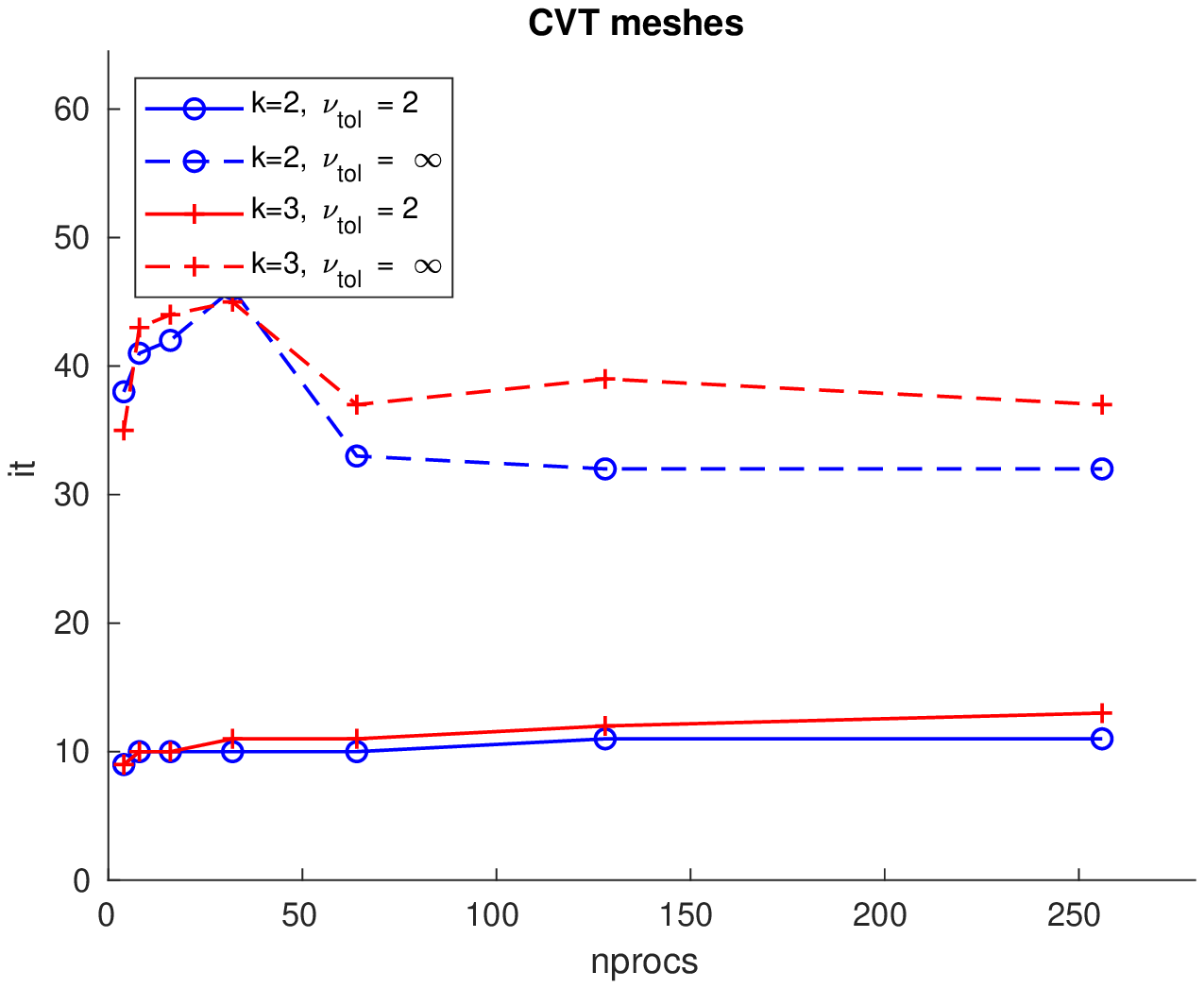}}
\subfloat{\label{fig:d}\includegraphics[width=0.49\textwidth]{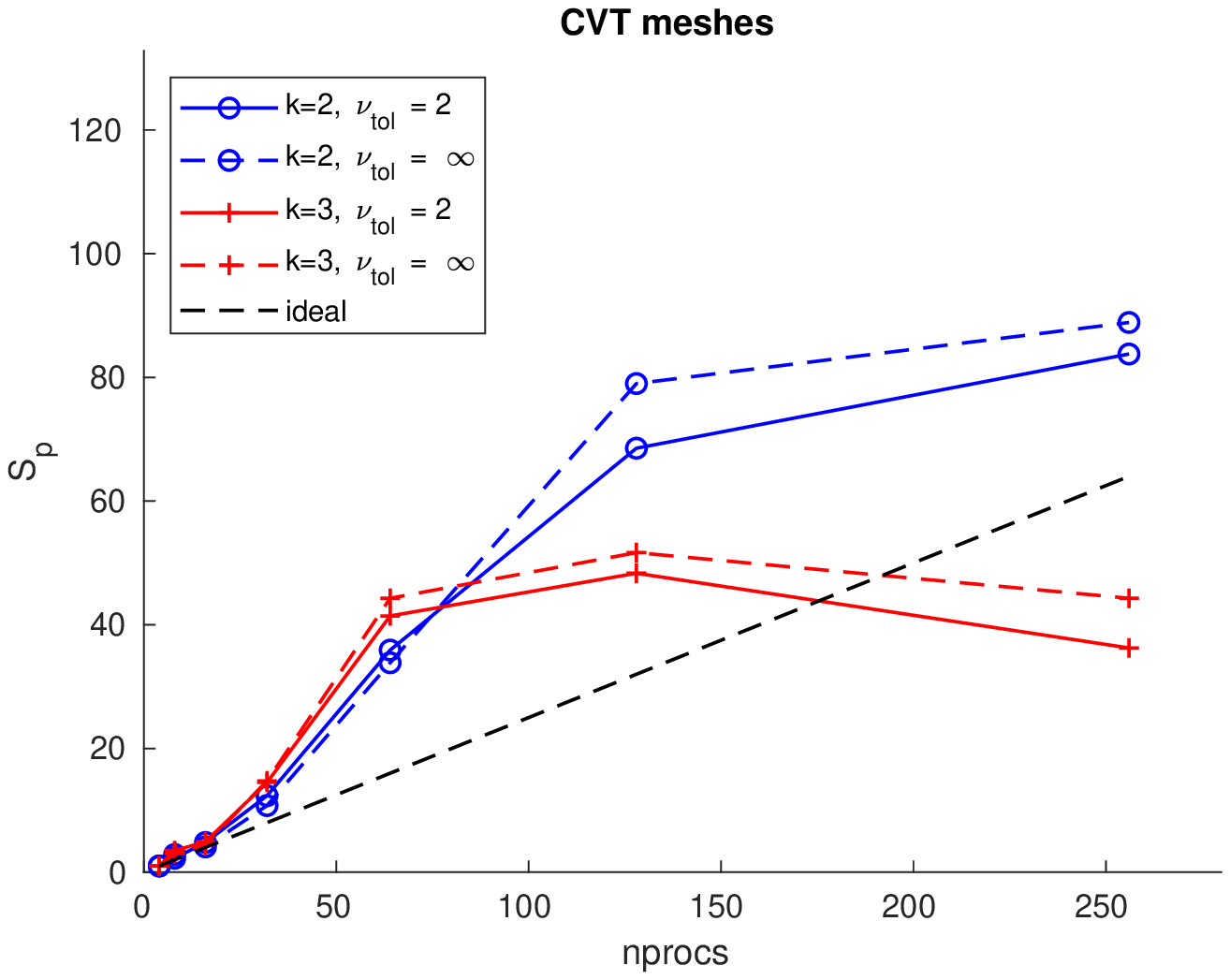}}\\
\subfloat{\label{fig:e}\includegraphics[width=0.49\textwidth]{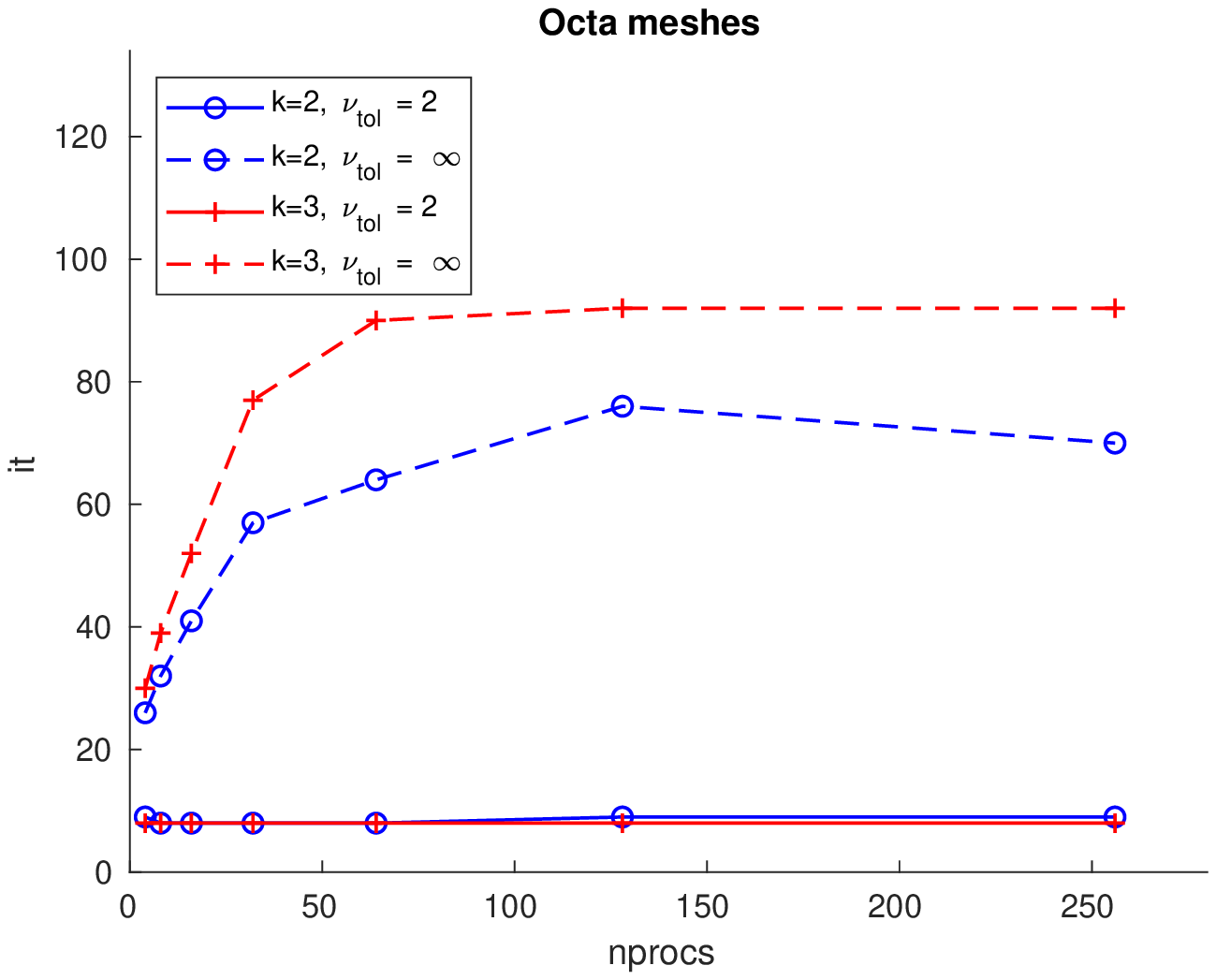}}
\subfloat{\label{fig:f}\includegraphics[width=0.49\textwidth]{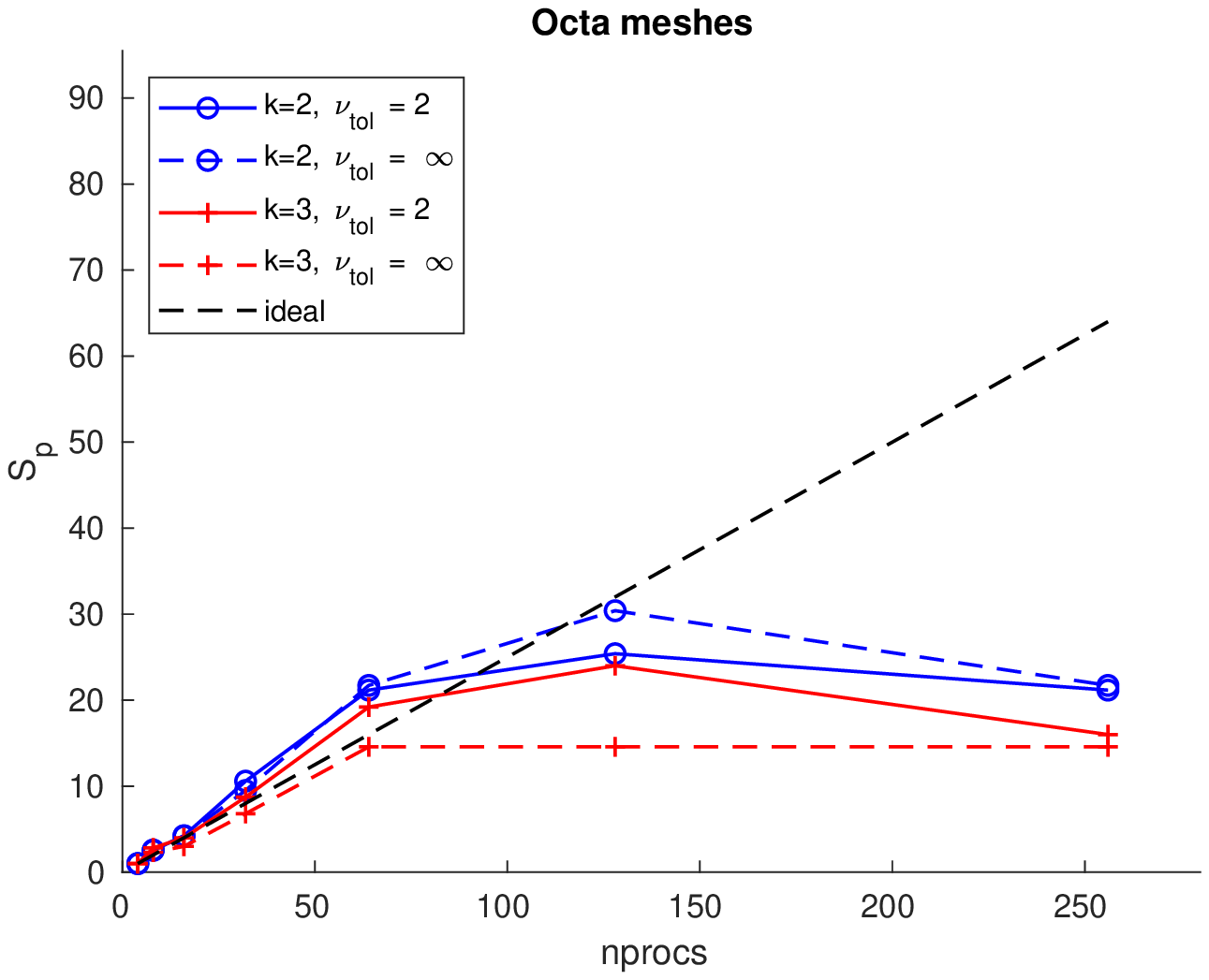}}

\caption{Test 1: strong scaling. Iteration of the CG (left) and parallel speedup (right) with the BDDC preconditioners for the two different primal spaces for different meshes and degrees of VEM discretizations.}
\label{fig:StrongScalability}
\end{figure}

\begin{table}{
    \footnotesize
    \captionsetup{position=top} 
    \begin{center}
        \begin{tabular}{|r r| r| r r r | r r r |}
    					\hline 
    					& & & \multicolumn{3}{c|}{$\nu_{tol} = 2$} & \multicolumn{3}{c|}{$\nu_{tol} = \infty$} \\
    					& nEl & nDofs &  $N_\Pi$ & it & $k_2$ & $N_\Pi$ & it & $k_2$ \\
    					\hline
                            \multirow{5}{*}{\STAB{\rotatebox[origin=c]{90}{CUBE}}}
    					& \numgru{4096} & \numgru{124195} & \numgru{2861} & 9 & 1.59	& 125 & 44 & 44.67 \\
    					& \numgru{8000} & \numgru{239763} & \numgru{3644} & 9 & 1.62 & 152 & 50 & 54.53 \\
    					& \numgru{13824} & \numgru{408243} & \numgru{4577} & 9 & 1.57 & 125 & 54 & 63.98 \\
    					& \numgru{21952} & \numgru{643387} & \numgru{5257} & 9 & 1.64 & 152 & 58 & 72.37 \\
    					& \numgru{32768} & \numgru{954947} & \numgru{6141} & 9 & 1.57 & 125 & 62 & 80.63 \\
    					\hline
    					\hline 
    					\multirow{5}{*}{\STAB{\rotatebox[origin=c]{90}{CVT}}}
    					& 125 & \numgru{8945} & \numgru{1192} & 12 & 2.29& 633 & 22 & 7.7 \\
    					& \numgru{1000} & \numgru{76051} & \numgru{3666} & 10 & 2.01 & 801 & 32 & 16.09 \\
    					& \numgru{2000} & \numgru{154067} & \numgru{5018} & 10 & 1.93 & 829 & 40 & 28.35 \\
    					& \numgru{4000} & \numgru{311155} & \numgru{6700} & 10 & 1.92 & 836 & 46 & 37.42 \\
    					& \numgru{8000} & \numgru{626455} & \numgru{8890} & 10 & 1.90 & 833 & 56 & 53.10 \\
    					\hline
                            \hline
                            \multirow{5}{*}{\STAB{\rotatebox[origin=c]{90}{OCTA}}}
    					& 576 & \numgru{22035} & \numgru{1313} & 9 & 1.76 & 125 & 36 & 26.15 \\
    					& \numgru{4608} & \numgru{166179} & \numgru{3065} & 8 & 1.52 & 125 & 40 & 37.19 \\
    					& \numgru{9000} & \numgru{320763} & \numgru{4571} & 9 & 1.64 & 157 & 57 & 47.31 \\
    					& \numgru{15552} & \numgru{549939} & \numgru{4793} & 8 & 1.53 & 125 & 57 & 70.78 \\
    					& \numgru{30375} & \numgru{1065693} & \numgru{8513} & 10 & 1.85 & 474 & 57 & 46.62 \\
    					\hline
    				\end{tabular}
    	\end{center}}
     \vspace{1mm}
     \caption{Test 2. Optimality test with respect to the mesh size, with $k = 2$ and procs = 32.}\label{tab:optMesh1}
\end{table}

\begin{table}
    {\footnotesize
    \captionsetup{position=top} 
    \begin{center}
            		\begin{tabular}{|r r|r| r r r | r r r |}
    				    \hline 
    					& & & \multicolumn{3}{c|}{$\nu_{tol} = 2$} & \multicolumn{3}{c|}{$\nu_{tol} = \infty$} \\
    					& k & nDofs &  $N_\Pi$ & it & $k_2$ & $N_\Pi$ & it & $k_2$ \\
                        \hline
                        \multirow{3}{*}{\STAB{\rotatebox[origin=c]{90}{CUBE}}}
    					& 2 & \numgru{16787} & \numgru{1273} & 9 & 1.81	& 125 & 34 & 23.86 \\
    					& 3 & \numgru{40667} & \numgru{2683} & 8 & 1.44 & 125 & 49 & 46.34 \\
    					& 4 & \numgru{76387} & \numgru{4495} & 8 &  1.49 & 125 & 64 & 80.67 \\
    				\hline
                        \hline
                        \multirow{3}{*}{\STAB{\rotatebox[origin=c]{90}{CVT}}}
    					& 2 & \numgru{8945} & \numgru{1192} & 12 & 2.29	& 633 & 22 & 7.7 \\
    					& 3 & \numgru{19256} & \numgru{2444} & 12 & 2.26 & 633 & 31 & 14.89 \\
    					& 4 & \numgru{33487} & \numgru{4604} & 14 & 3.02 & 633 & 46 & 28.86 \\
    				\hline
                        \hline
                        \multirow{3}{*}{\STAB{\rotatebox[origin=c]{90}{OCTA}}}
    					& 2 & \numgru{22035} & \numgru{1313} & 9 & 1.76	& 125 & 36 & 26.15 \\
    					& 3 & \numgru{52251} & \numgru{2847} & 8 & 1.46 & 125 & 55 & 53.66 \\
    					& 4 & \numgru{96675} & \numgru{4951} & 9 &  1.61 & 125 & 75 & 91.81 \\
    					\hline
    				\end{tabular}
    	\end{center}}
     \vspace{1mm}
     \caption{Test 3. Optimality Test Increasing the polynomial degree $k$ with procs = 32 and number of elements for CUBE = $512$, CVT = $125$ and OCTA = $576$. } \label{tab:optDeg1}
\end{table}

\begin{table}
    {\footnotesize
    \captionsetup{position=top} 
    \begin{center}
            	\begin{tabular}{|r r|r|r|r|r r|r r|}
			        \hline
			         & & & & MUMPS & \multicolumn{2}{c|}{Block-Schur} & \multicolumn{2}{c|}{BDDC} \\
					 & nEl & k & nDofs & $T_{sol}$ & it & $T_{sol}$ & it & $T_{sol}$ \\
                    \hline
                    \multirow{3}{*}{\STAB{\rotatebox[origin=c]{90}{CUBE}}}
					& \numgru{32768}  & 2 & \numgru{954947}  & 297 & 705 & 110 & 21 & 26\\
					& \numgru{13824}  & 3 & \numgru{1009803} & 416 & NC & NC & 18 & 22\\
     			    & \numgru{8000}  & 4 & \numgru{1119523} & 465 & NC &       NC & 13 & 34\\
				\hline
                    \hline
                    \multirow{3}{*}{\STAB{\rotatebox[origin=c]{90}{CVT}}}
					& \numgru{8000}  & 2 & \numgru{627455} & 913 & 568 & 134 & 16 & 58 \\
					& \numgru{4000}  & 3 & \numgru{666301} & 971 & NC & NC & 17 & 103\\
     			    & \numgru{2000}  & 4 & \numgru{571696} & 842 & NC     & NC & 22 & 93\\
				\hline
                    \hline
                    \multirow{3}{*}{\STAB{\rotatebox[origin=c]{90}{OCTA}}}
					& \numgru{30375}  & 2 & \numgru{1065693} & 285 & 893 & 180 & 21 & 60\\
					& \numgru{15552}  & 3 & \numgru{1322571} & 355 & NC & NC & 19 & 73\\
                        & \numgru{9000}  & 4 & \numgru{1436523} & 548 & NC & NC & 34 & 97\\
					\hline
				\end{tabular}			
	\end{center}}
        \vspace{1mm}
        \caption{Test 4. Solver Comparision. Performance comparision among different parallel solver with procs = $64$.}\label{solverComparision}
\end{table}

\begin{table}
    {\footnotesize
    \captionsetup{position=top} 
    \begin{center}
            	\begin{tabular}{|r r|r r|r r|r r|r r|}
			        \hline
			         \multicolumn{2}{|c|}{$DR(\nu)$}  & \multicolumn{2}{c|}{$1e+0$} & \multicolumn{2}{c|}{$1e+2$} & \multicolumn{2}{c|}{$1e+4$} & \multicolumn{2}{c|}{$1e+6$} \\
					 & n & it & $k_2$ & it & $k_2$ & it & $k_2$ & it & $k_2$  \\
                    \hline
				\multirow{4}{*}{\STAB{\rotatebox[origin=c]{90}{CUBE}}}
                        & 1  & 18 & 6.2 & 19 & 7.7 & 20 & 7.6 & 19 & 7.7 \\
                        & 5  & 17 & 6.6 & 19 & 7.7 & 19 & 7.6 & 19 & 8.4 \\
                        & 10 & 18 & 6.6 & 19 & 7.6 & 19 & 7.6 & 19 & 8.5 \\
                        & 20 & 17 & 6.6 & 19 & 7.4 & 19 & 7.7 & 19 & 6.1 \\
                    \hline
                    \hline
				\multirow{4}{*}{\STAB{\rotatebox[origin=c]{90}{CVT}}}
                        & 1  & 15 & 4.7 & 15 & 4.8 & 16 & 5.2 & 17 & 6.1 \\
                        & 5  & 15 & 4.7 & 15 & 4.8 & 17 & 5.8 & 19 & 8.1 \\
                        & 10 & 15 & 4.7 & 15 & 5.1 & 17 & 6.1 & 20 & 8.1 \\
                        & 20 & 14 & 4.7 & 17 & 7.7 & 20 & 9.7 & 28 & 15.2 \\
                    \hline
                    \hline
				\multirow{4}{*}{\STAB{\rotatebox[origin=c]{90}{OCTA}}}
                        & 1  & 17 & 6.2 & 19 & 7.4 & 19 & 7.6 & 20 & 8.4 \\
                        & 5  & 17 & 6.2 & 18 & 7.2 & 19 & 7.6 & 19 & 8.4 \\
                        & 10 & 17 & 6.2 & 18 & 6.8 & 20 & 8.8 & 19 & 8.9 \\
                        & 20 & 17 & 6.2 & 18 & 7.0 & 19 & 7.6 & 18 & 8.4 \\
                    \hline
				\end{tabular}
    	\end{center}}
     \vspace{1mm}
        \caption{Test 5. Multi-sinker benchmark problem with procs = 64 and number of elements for CUBE = $\numgru{13824}$, CVT = $\numgru{4000}$ and OCTA = $\numgru{15552}$. }\label{sinkers}
\end{table}

\section{Numerical Results}\label{sec:numExe}
In this Section, we report the numerical results to validate our theoretical estimates of the BDDC algorithm for solving the Stokes model problem \eqref{VarForm}. In particular we solve a problem on the unit cube $[0,1]^3$ with a known solution (Figure \ref{fig:solPlot}) imposing Neumann boundary conditions on two faces of the cube and homogeneous Dirichlet boundary conditions on the other ones. The BDDC method is used as a preconditioner for system \eqref{globInt}, which is solved by the CG method with a stopping criterion of a $10^{-8}$ reduction of the $l^2-$norm of the relative residual. We consider three types of meshes: hexahedral (Cube), octahedral (Octa), and Voronoi (CVT), see Figure~\ref{fig:mesh}.
Our distributed memory implementation is based on the PETSc library \cite{balay2019petsc}. We refer to \cite{zampini2016pcbddc} for the details related to the BDDC implementation in PETSc.
In our experiments, we compare two different choices of primal spaces, corresponding to tolerance $\nu_{tol}=2$ and $\nu_{tol}=\infty$. The first one represents the adaptive coarse space built to keep the condition number under the fixed tolerance $\nu_{tol}=2$. The latter represents the minimal coarse spaces created as explained in Section~\ref{sec::4} to satisfy the two Assumption~\ref{ass1} and~\ref{ass2}. We also compare the BDDC algorithms against our previous block-diagonal preconditioner \cite{dassiS.2020b} and the parallel direct solver MUMPS \cite{amestoy2001fully, amestoy2006hybrid}.
We conclude by testing the robustness of our adaptive BDDC algorithm on a benchmark problem with variable viscosity.
All the numerical tests presented in the following have been performed on the Linux cluster INDACO (www.indaco.unimi.it) of the University of Milan, constituted by 16 nodes, each carrying 2 INTEL XEON E5-2683V4 processors at 2.1 GHz, with 16 cores each. 

In the tables we use the following notation: procs = number of CPUs, nEl~=~number of VEM elements, k = degree of VEM approximation, nDofs = number of dofs, $N_{\Pi}$ = number of primal constraints, it = iteration count (GMRES for Block-Schur, CG for BDDC), $k_2$ = conditioning number,  $T_{ass}$= time to assemble the stiffness matrix and the right-hand side, $T_{sol}$ = time to solve the interface saddle point problem and $S^{id}$ = ideal speed up, $S_p$ = parallel speed up.

\subsection{Test 1: strong scaling}
We first study the strong scalability of our solvers. We keep fixed the global number of the dofs and the degree of the VEM approximation $k$, while we increase the number of processors from 4 to 256. We consider CUBE mesh with $\numgru{408243}$ dofs, a CVT mesh with $\numgru{311155}$ dofs and an OCTA mesh with $\numgru{549939}$ dofs. We recall that denoting by $p$ the number of processors, the parallel speedup is defined as:
\begin{align*}
    S_p := \frac{\text{CPU time with 4 processors}}{\text{CPU time with p processors}}.   
\end{align*}
In Table \ref{strongScal}, we report the results related to the three polyhedral meshes with $k=2$. In Figure \ref{fig:StrongScalability}, we plot the number of iterations and the parallel speedup for the case $k=3$.
We observe that the CPU time $T_{ass}$, needed to assemble the stiffness matrix and the right-hand-side is scalable, with a speedup very close to the ideal ones.
The adaptive BDDC method ($\nu_{tol} = 2$) results scalable since the number of CG iterations remains bounded and the solution time decreases as the number of the processors increases. We note that, as usual in a strong scalability study, the parallel speedup does not increase when the number of processors is large with respect to the local size of the problems. This is because communication time overcomes the time for computation.
Also, the minimal coarse space results are scalable for the degree $k=2$ and $3$, with the same behavior as the adaptive BDDC.

\subsection{Test 2: optimality test with respect to the mesh size}
We now perform an optimality test with respect to the mesh size: we keep fixed the number of processors at 32 and we increase the number of dofs, maintaining the degree of the VEM discretization $k=2$. The results are reported in Table \ref{tab:optMesh1}. We observe that the adaptive solver has an optimal behavior irrespective of the type of polyhedral mesh considered since the iteration count does not grow with the refinement of the mesh. The minimal coarse space has quasi-optimal behavior since both the iteration count and the condition number exhibit a logarithmic growth as predicted by Theorem \ref{teoEst}. Similar results also occur for the cases $k=3$ and $4$.

\subsection{Test 3: optimality test with respect to the polynomial degree}
In this test, we study the robustness of our preconditioners when increasing the polynomial degree of the VEM discretization. The tests are performed keeping fixed the number of processors again at 32 and the mesh size. The results reported in Table \ref{tab:optDeg1} show that the adaptive BDDC algorithm is robust with respect to the polynomial degree in all meshes. The BDDC solver with minimal coarse space instead exhibits a slight increase of the condition number and iterations count when the degree $k$ increases.

\subsection{Test 4: solvers comparison}
In Table \ref{solverComparision}, we compare the performance of the CG method accelerated by the adaptive BDDC preconditioner against the direct solver MUMPS and the GMRES method accelerated by the Block-Schur preconditioner proposed in \cite{dassi2020parallel}. The latter preconditioner is of the form:
\begin{equation*}\label{bsprec}
	B = 
	\left[
	\begin{array}{cc}
		diag(A)^{-1} & 0\\
		0 & \widetilde{S}^{-1}\\
	\end{array}
	\right]
\end{equation*}
where $\widetilde{S}=-B diag(A)^{-1} B^T$ is the approximate Schur complement of the system \eqref{MatrixForm} and the inversion of this matrix is performed by MUMPS.
We can see that the adaptive BDDC is significantly faster than the other solvers for all the meshes considered and for the three different degrees of the VEM discretization. We also note that the Block-Schur preconditioner is not robust for the degree $k=3$ and $4$, since the GMRES method does not converge (NC in the table).

\subsection{Test 5: Multi-sinker benchmark problem}
To consider a practical application, as in \cite{rudi2017weighted}, we conclude by testing the robustness of our adaptive BDDC algorithm on a benchmark problem with heterogeneous viscosity. We perform a multi-sinker test problem with inclusions of equal size placed randomly in the unit cube domain so that they can overlap and intersect the boundary. The viscosity coefficient $\nu(\mathbf{x})$ , is defined in terms of a $C^\infty$ indicator function $\chi_n(\mathbf{x})\in [0,1]$ that accumulates $n$ sinkers via the product of modified Gaussian functions, see \cite{rudi2017weighted} for more details about these functions. In this way, the viscosity exhibits sharp gradients, and its dynamic ratio $DR(\nu):=\nu_{max}/\nu_{min}$ in our study can be up to six orders of magnitude.
We fix the mesh element size and the number of processes at $64$, and we study the iteration count and condition number of the adaptive BDDC algorithm with $\mu_{tol} = 5$, varying the dynamic ratio $DR(\nu)$ from $1$ to $1e+6$, and the number of sinkers $n$ from $1$ to $20$.
The results reported in Table \ref{sinkers}, obtained on different polyhedral meshes and for a VEM discretization of degree $k=2$, show the robustness of our adaptive preconditioner since the number of iterations and the condition number do not grow when increasing the number of sinkers and the viscosity ratio.

\section{Conclusions}
We have analyzed BDDC preconditioners for the saddle-point linear system deriving from a divergence free VEM discretization of the steady three-dimensional Stokes equations. The numerical tests have validated the convergence estimates, showing the scalability and quasi-optimality of the algorithm, under appropriate choices of the coarse space. We have also investigated an adaptive technique to enrich the coarse space, based on deluxe scaling functions, that is more robust than the minimal coarse space with respect to the order of VEM approximation. We have also shown that the adaptive BDDC method outperforms in terms of CPU time  other competitive solvers and that it is robust on a challenging multi-sinker test case.

\section*{Acknowledgments}
We acknowledge the usage of the INDACO Linux Cluster of the University of Milan and the support of INDAM-GNCS. 

\bibliographystyle{siamplain}
\bibliography{literature}
\nocite{pavarino2002balancing}
\nocite{goldfeld2003balancing}
\end{document}